\documentclass[oneside, a4paper,reqno]{amsart}
\usepackage{amsmath, amsthm, amscd, amssymb, latexsym, eucal}
\usepackage{pdfsync}
\usepackage{bbm}
\usepackage{bm}
\usepackage{stmaryrd}
\usepackage{mathrsfs}
\usepackage{hyperref}
\usepackage{tikz-cd}
\usepackage[all]{xy}

 \calclayout \makeatletter
 \calclayout \makeatletter
\def\serieslogo@{} \def\@setcopyright{} \makeatother
\makeatletter
\renewcommand*\env@matrix[1][c]{\hskip -\arraycolsep
  \let\@ifnextchar\new@ifnextchar
  \array{*\c@MaxMatrixCols #1}}
\makeatother

\usepackage{multienum}

\usepackage[colorinlistoftodos]{todonotes}

\usepackage{hyperref}
\usepackage{color}
\usepackage[nospace,noadjust]{cite}

\hypersetup{colorlinks=true,
     breaklinks=true,
     linkcolor=blue,
     citecolor=red,    
     bookmarks=true,
     urlcolor= black,
     pageanchor=true
     }

\newtheorem{theorem}{Theorem}[section]
\newtheorem{lemma}[theorem]{Lemma}
\newtheorem{corollary}[theorem]{Corollary}
\newtheorem{proposition}[theorem]{Proposition}

\newtheorem{Theox}{Theorem}

\newtheorem{Corox}[Theox]{Corollary}

\theoremstyle{definition}
\newtheorem{definition}[theorem]{Definition}

\newtheorem{example}[theorem]{Example}

\theoremstyle{remark}
\newtheorem{remark}[theorem]{Remark}

\DeclareMathOperator{\Hom}{Hom}

\newcommand{\Colim}{\varinjlim}

\newcommand{\Mod}{\textnormal{Mod}}
\newcommand{\mods}{\textnormal{mod}}

\newcommand{\Ab}{\textbf{Ab}}

\numberwithin{equation}{section}
\def \-{\text{-}}
\def \k{\mathbbm{k}}
\def \Der{\textbf{D}}
\def \Com{\textbf{C}}

\def \fp{\text{fp}}

\def \Im{\text{Im}}

\def \ev{\text{ev}}
\def \coev{\text{coev}}

\begin{document}
\vspace{10pt}
\title{Hopf algebras are determined by their monoidal derived categories}
\author{Yuying Xu, Junhua Zheng}
\thanks{Both authors have received support by Programs of China Scholarship Council (first auther: No. 202206190132, second author: No. 202307300022)}

\address{}
\address{Y. Y. Xu}
\address{School of Mathematics, Nanjing University, 210093 Nanjing, People's Republic of China}
\address{Institute of Algebra and Number Theory, University of Stuttgart, Pfaffenwaldring
57, 70569 Stuttgart, Germany}
\email{Yuying.Xu@mathematik.uni-stuttgart.de}
\address{}
\address{J. H. Zheng}
\address{Institute of Algebra and Number Theory, University of Stuttgart, Pfaffenwaldring
57, 70569 Stuttgart, Germany}
\email{Junhua.Zheng@mathematik.uni-stuttgart.de}

\maketitle

\begin{abstract}
We show that two finite-dimensional Hopf algebras are gauge equivalent if and only if their bounded derived categories are monoidal triangulated equivalent. More generally, a monoidal derived equivalence between locally finite tensor Grothendieck categories induces a monoidal abelian equivalence.
\end{abstract}

\setcounter{tocdepth}{2} \tableofcontents

\section{Introduction}

In representation theory there is a hierarchy of ways to compare module categories, from the very strong Morita equivalence via tilting to the very general derived equivalence.  When module categories or abelian categories in addition carry tensor structures, an analogous hierarchy seems to go from gauge equivalences  to \textit{monoidal derived equivalences} that are monoidal triangulated equivalences of bounded derived categories. Here gauge equivalences are special cases of Morita equivalences for Hopf algebras, which are Morita equivalences satisfying additional conditions on coalgebras' structures. However, for finite-dimensional Hopf algebras it turns out that, unexpectedly, monoidal derived equivalences induce gauge equivalences and thus are not more general than these special Morita equivalences.

For finite-dimensional Hopf algebras the main result of this article is the following reconstruction theorem: 

\begin{Corox} \emph{(Remark \ref{corollary A})}\label{Corollary A}
Let $H$ and $H'$ be finite-dimensional Hopf algebras. Then the following are equivalent:
\begin{enumerate}
\item $H$ and $H'$ are gauge equivalent;
\item $\mods\-H$ and $\mods\-H'$ are monoidal abelian equivalent;
\item $\Mod\-H$ and $\Mod\-H'$ are monoidal abelian equivalent;
\item The derived categories $\textnormal{\Der}^{b}(\mods\-H)$ and $\textnormal{\Der}^{b}(\mods\-H')$ are monoidal triangulated equivalent;
\item The derived categories $\textnormal{\Der}^{b}(\Mod\-H)$ and $\textnormal{\Der}^{b}(\Mod\-H')$ are monoidal triangulated equivalent.

\end{enumerate}

\end{Corox}
More generally, we discuss the case of a \textit{locally finite tensor Grothendieck category} $\mathcal{A}$ which is a locally finitely presented Grothendieck category such that the full subcategory $\textnormal{\fp}(\mathcal{A})$ consisting of finitely presented objects forms a finite tensor category. In this article, a \textit{tensor category} is a rigid monoidal abelian category in the sense of \cite[Definition 4.1.1]{etingof2016tensor}, and a \textit{finite abelian category} is a length category such that there are finitely many isomorphism classes of simple objects. Corollary \ref{Corollary A} is a special case of the following general result. 

\begin{Theox} \emph{(Theorem \ref{4 eq})}\label{Theorem B}
Let $\mathcal{A}$ and $\mathcal{B}$ be locally finite tensor Grothendieck categories. Then the following are equivalent:
\begin{enumerate}
\item $\textnormal{\fp}(\mathcal{A})$ and $\textnormal{\fp}(\mathcal{B})$ are monoidal abelian equivalent;
\item $\mathcal{A}$ and $\mathcal{B}$ are monoidal abelian equivalent;
\item $\textnormal{\Der}^{b}(\textnormal{\fp}(\mathcal{A}))$ and $\textnormal{\Der}^{b}(\textnormal{\fp}(\mathcal{B}))$ are monoidal triangulated equivalent;
\item $\textnormal{\Der}^{b}(\mathcal{A})$ and $\textnormal{\Der}^{b}(\mathcal{B})$ are monoidal triangulated equivalent.

\end{enumerate}

\end{Theox}
The large module category $\Mod\-H$ of a finite-dimensional (weak quasi-)Hopf algebra $H$ is a locally finite tensor Grothendieck category, since $\textnormal{\fp}(\Mod\-H)=\mods\-H$ is a finite tensor category \cite{haring1997reconstruction}. In fact, Corollary \ref{Corollary A} holds for all finite-dimensional (weak quasi-)Hopf algebras.

It is unusual that the derived category of an abelian category determines the abelian category. Results showing under some assumptions that this happens are known as reconstruction theorems. A reconstruction theorem in algebraic geometry, due to Bondal and Orlov \cite{bondal2001reconstruction}, shows for a smooth irreducible projective variety $X$ with ample canonical or anticanonical sheaf, if the bounded derived category $\Der^{b}_{\text{coh}}(X)$ is equivalent to $\Der^{b}_{\text{coh}}(X')$ as a graded category for a smooth algebraic variety $X'$, then $X$ and $X'$ are isomorphic. This does not use a monoidal structure, but a grading. Balmer showed that the derived category of coherent sheaves on a smooth variety, when considered as a monoidal triangulated category, completely determines the variety \cite{balmer2002presheaves}. As a corollary, a monoidal derived equivalence of the perfect complexes over two reduced noetherian schemes induces an isomorphism between these two schemes. Aihara and Mizuno showed that a preprojective algebra of Dynkin type is derived equivalent only to itself, up to Morita equivalence \cite{aihara2017classifying}. Zhang and Zhou in \cite{zhang2022frobenius} proved a reconstruction theorem for finite-dimensional weak bialgebras, assuming that these are hereditary.

The big difference between derived Morita theory of algebras and Theorem \ref{Theorem B} can be seen in the behaviour of the main tool we use to prove Theorem \ref{Theorem B}. Instead of t-structures, introduced by Beilinson, Bernstein and Deligne in \cite{beilinson1982perverse} and used by Alonso Tarr{\'\i}o, Jeremías L{\'o}pez and Souto Salorio to give an alternative proof of Rickard's theorem \cite{alonso2003construction}, in this article monoidal t-structures are introduced, modifying mtt-structures defined in \cite{zhang2022frobenius}. In the situation of interest here, these are shown to be tensor reduced. In Corollary \ref{yy main1} it is shown that hearts of monoidal t-structures are monoidal abelian categories. Corollary \ref{dual prop} strengthens this by showing that hearts of tensor reduced monoidal t-structures are even tensor categories. And then it turns out that monoidal triangulated equivalences induce monoidal equivalences between hearts - a statement that fails completely in the absence of a tensor structure. 

An outline of this paper is as follows.

In Section \ref{t-stru}, we present the definitions of \textit{monoidal t-structures} and their corresponding deviation (see Definition \ref{definition of monoidal t-structure}) on monoidal triangulated categories.  The discussion of tensor reducedness (see Definition \ref{reduced monoidal})  helps us to prove that the equivalent tensor reduced monoidal t-structures are equal (see Theorem \ref
{unique u.t.e}), which implies that their hearts are equivalent as monoidal abelian categories (see Corollary \ref{unique u.t.e cor}). Moreover, the cases of locally finite tensor Grothendieck categories and finite tensor categories fit into the setup of Corollary \ref{unique u.t.e cor}, which helps us prove our main results, as discussed in Section \ref{L.N.T.R.G} and Section \ref{Morita/Rickard eq}.

\par  Section \ref{L.N.T.R.G} is devoted to introduce the concept of locally finite tensor Grothendieck categories (see Definition \ref{L.N.T.G cat}). For a finite-dimensional Hopf algebra $H$, $\mods\-H$ is a tensor category, whereas $\Mod\-H$ is not rigid. We are going to generalize the results from tensor categories to more general monoidal abelian categories by exploring locally finite tensor Grothendieck categories as a categorical version of $\Mod\-H$. As a generalized version of Morita equivalences, we prove that locally finite tensor Grothendieck categories $\mathcal{A}$ and $\mathcal{B}$ are monoidal abelian equivalent if and only if $\fp(\mathcal{A})$ and $\fp(\mathcal{B})$ are monoidal abelian equivalent (see Proposition \ref{Morita tensor eq}). A special case of this statement yields a corresponding result for modules over finite-dimensional Hopf algebras. Furthermore, we also prove that locally finite tensor Grothendieck categories and finite tensor categories are tensor reduced (see Proposition \ref{trivial ideal} and Theorem \ref{0tensor lemma}).

\par In Section \ref{Morita/Rickard eq}, we apply the results of Section \ref{t-stru} and \ref{L.N.T.R.G} to prove our main results, Theorem \ref{Theorem B} and Corollary \ref{Corollary A}. Moreover, we consider the stable categories of finite tensor categories with enough projectives called \textit{monoidal stable categories} which are Frobenius categories. Then a monoidal version of a theorem of Rickard concerning stable categories is given (see Corollary \ref{result4}).

\subsection*{Notation}
Throughout this paper, $\k$ is assumed to be an algebraically closed field. All categories are $\mathbbm{k}$-linear categories. For any ring $A$, the category of all right A-modules is denoted by $\Mod\-A$, and the category of finitely generated right $A\-$modules is denoted by $\mods\-A$. Unless otherwise stated, all the modules considered in this paper are right modules and the word subcategory stands for a full and strict subcategory.

\section{Monoidal t-structures on monoidal triangulated categories}\label{t-stru}

\subsection{Tensor categories and monoidal triangulated categories}

We firstly recall some definitions and properties related to monoidal categories (in the sense of \cite[Definition 2.2.8]{etingof2016tensor}). Readers are referred to \cite{etingof2016tensor} for more details. Afterwards, our main tools monoidal triangulated categories, especially monoidal derived categories, get introduced and investigated.

\par Let $(\mathcal{C}, \otimes, \mathbbm{1})$ be a monoidal category. An object $X^{\ast}$ in $\mathcal{C}$ is said to be a \textbf{left dual} of $X$ if there exist morphisms $\ev_{X}:X^{\ast} \otimes X \rightarrow \mathbbm{1}$ and $\coev_{X}: \mathbbm{1} \rightarrow X \otimes X^{\ast}$, called the \textbf{evaluation} and \textbf{coevaluation}, such that the following compositions 
$$X \stackrel{\coev_{X} \otimes \text{id}_{X}}{\xrightarrow{\hspace*{1.5cm}}} X\otimes X^{\ast} \otimes X \stackrel{\text{id}_{X}\otimes \ev_{X}}{\xrightarrow{\hspace*{1.5cm}}} X\;\;\text{and}\;\;X^{\ast} \stackrel{\text{id}_{X^{\ast}}\otimes \coev_{X}  }{\xrightarrow{\hspace*{1.5cm}}} X^{\ast}\otimes X \otimes X^{\ast} \stackrel{\ev_{X}\otimes \text{id}_{X^{\ast}} }{\xrightarrow{\hspace*{1.5cm}}} X^{\ast}$$
are the identity morphisms. Dually, we can define a \textbf{right dual} ${^{\ast}X}$ of $X$. If $X\in\mathcal{C}$ has a left (resp. right) dual object, then it is unique up to unique isomorphism (see \cite[Proposition 2.10.5]{etingof2016tensor}). An object in a monoidal category is called $\textbf{rigid}$ if it has left and right duals. A monoidal category \textbf{has left duals} (resp. \textbf{has right duals}) if every object has a left (resp. right) dual. Moreover, a monoidal category is called $\textbf{rigid}$ if every object is rigid.

\begin{remark}(\cite[Exercise 2.10.6]{etingof2016tensor}) \label{rigid under monoidal}
    Let $F:\mathcal{C}\rightarrow\mathcal{C'}$ be a monoidal functor between two monoidal categories. If $X$ is an object in $\mathcal{C}$ with a left dual $X^{\ast}$, then $F(X^{\ast})$ is a left dual of $F(X)$. The same result holds for right duals.
\end{remark}

\begin{definition}(\cite[Definition 4.1.1]{etingof2016tensor})\label{tensor abelian cat}
A rigid monoidal abelian category $(\mathcal{C}, \otimes, \mathbbm{1})$ is called a \textbf{multitensor category} if the bifunctor $\otimes: \mathcal{C}\times \mathcal{C}\rightarrow \mathcal{C}$ is bilinear on morphisms. If in addition $\text{End}_{\mathcal{C}}(\mathbbm{1})\cong \k$, then $\mathcal{C}$ is called a \textbf{tensor category}.
\end{definition}
The bifunctor $\otimes$ in a multitensor category is exact in each variable (i.e., biexact) \cite[Proposition 4.2.1]{etingof2016tensor}. Note that here we do not require the condition ``locally finite" as in \cite[Definition 4.1.1]{etingof2016tensor} since it is not necessary in this whole section.

\begin{example}\label{hopf algebra}
Roughly speaking, a Hopf algebra is a bialgebra carrying an antipode $S$. The reader is referred to \cite{montgomery1993hopf} for details on Hopf algebras. Considering a  finite-dimensional Hopf algebra $H$ over $\mathbbm{k}$, the category $\mods\-H$ is a monoidal categoy, with $\otimes_{\k}$ being the tensor product of $H\-$modules over $\k$ and the unit object $\k$. Moreover, for any right $H\-$module $X$, there are two actions of $H$ on the $\k\-$linear dual space $X^*$ by using the antipode $S$ and $S^{-1}$. These two different actions on $X^*$ turn $X^*$ into the right and left dual of $X$ respectively. In conclusion, ($\mods\-H, \otimes_{\k}, \k)$ is a tensor category.
\end{example}

According to \cite{nakano2022noncommutative}, a \textbf{monoidal triangulated category} $(\mathcal{T},\otimes,\Sigma,\mathbbm{1})$ is a triangulated category with shift functor $\Sigma$, which is also a monoidal category with  unit object $\mathbbm{1}$, such that the bifunctor $\otimes$ is exact in each variable. This involves isomorphisms $e_{X, Y}: \Sigma(X)\otimes Y\cong \Sigma(X\otimes Y) $ and $\theta_{X, Y}: X\otimes \Sigma(Y)\cong \Sigma(X\otimes Y)$, which are natural in any $X$, $Y\in\mathcal{T}.$  A \textbf{monoidal triangulated functor} is a triangulated functor respecting the monoidal structures and sending the unit to the unit. Two monoidal triangulated categories $\mathcal{T}$ and $\mathcal{T}'$ are said to be \textbf{monoidal triangulated equivalent} if there is a monoidal triangulated functor inducing an equivalence between $\mathcal{T}$ and  $\mathcal{T}'$.

Our definition of monoidal triangulated categories does not require further assumptions that sometimes are used in the literature. For example, there is literature (see \cite{balmer2010tensor, zhang2022frobenius}) requiring the natural isomorphisms $e_{-,-}$ and $\theta_{-,-}$ satisfying the anti-commuting diagram given in the definition of a suspended monoidal category \cite[Definition A.2.1]{mariano2004hilton}. In addition, the authors in \cite{balmer2010tensor, balmer2005spectrum} consider symmetric monoidal structures on monoidal triangulated categories called tensor triangulated categories.

\begin{example}\label{monoidal tri cat1}
For a monoidal abelian category $(\mathcal{A}, \otimes, \mathbbm{1})$ with biexact tensor product, there is a monoidal structure on the category of bounded chain complexes $\textbf{C}^{b}(\mathcal{A})$. Namely, for $X^{\bullet}, Y^{\bullet} \in \textbf{C}^{b}(\mathcal{A})$, $X^{\bullet}\widetilde{\otimes} Y^{\bullet}$ is defined to be the total complex
$$(X^{\bullet}\widetilde{\otimes} Y^{\bullet})^{n}:=\coprod_{p+q=n} X^{p}\otimes Y^{q}$$
with differentials
$$d^{n}_{X^{\bullet}\widetilde{\otimes} Y^{\bullet}}:=\sum_{p+q=n}(d^{p}_{X}\otimes \text{id}_{Y^{q}}+(-1)^{p}\text{id}_{X^{p}}\otimes d^{q}_{Y}).$$
Then $(\textbf{C}^{b}(\mathcal{A}), \widetilde{\otimes}, \mathbbm{1}^{\bullet})$ becomes a monoidal abelian category with biexact monoidal structure, where $\mathbbm{1}^{\bullet}$ is the stalk complex with $\mathbbm{1}$ concentrated in degree $0$. It is clear that the bounded homotopy category $(\textbf{K}^{b}(\mathcal{A}), \widetilde{\otimes}, \mathbbm{1}^{\bullet})$ is also a monoidal category whose monoidal structure $\widetilde{\otimes}$ is inherited from $\textbf{C}^{b}(\mathcal{A})$. Since $\widetilde{\otimes}$ preserves null-homotopy, $\textbf{K}^{b}(\mathcal{A})$ is a monoidal triangulated category. Furthermore, Acyclic Assembly Lemma in \cite [Lemma 2.7.3]{weibel1994homological} tells us that $\widetilde{\otimes}$ also preserves quasi-isomorphisms, which shows that the bounded derived category $(\textbf{D}^{b}(\mathcal{A}), \widetilde{\otimes}, \mathbbm{1}^{\bullet})$ is a monoidal triangulated category whose monoidal structure $\widetilde{\otimes}$ is inherited from $\textbf{K}^{b}(\mathcal{A}).$
\end{example}

\begin{remark}
     In this paper, the derived categories carrying monoidal structures as in Example \ref{monoidal tri cat1} are called \textbf{monoidal derived categories}. Monoidal triangulated equivalences between bounded derived categories will be simply called \textbf{monoidal derived equivalences}.
\end{remark}

\subsection{t-structures}
In this subsection, we will recall some basic definitions and properties related to t-structures on triangulated categories. Let $(\mathcal{T},\Sigma)$ be a triangulated category where $\Sigma$ is the translation functor.

\begin{definition}(\cite[Definition 1.3.1]{beilinson1982perverse})
A pair of full subcategories $\mathbbm{t}=(\mathcal{D}^{\leqslant 0},\mathcal{D}^{\geqslant 1})$ in $\mathcal{T}$ is said to be a \textbf{t-structure} on $\mathcal{T}$, if $\mathcal{D}^{\leqslant 0}, \mathcal{D}^{\geqslant 1}$ satisfy the following conditions:
\begin{itemize}
\item[(T1)] $\Sigma\mathcal{D}^{\leqslant 0} \subseteq \mathcal{D}^{\leqslant 0}$ and $\mathcal{D}^{\geqslant 1} \subseteq \Sigma \mathcal{D}^{\geqslant 1}$;
\item[(T2)] $\text{Hom}_{\mathcal{T}}(\mathcal{D}^{\leqslant 0},\mathcal{D}^{\geqslant 1})=0$;
\item[(T3)] For any object $X \in \mathcal{T}$, there is a distinguished triangle 

$$X^{\leqslant 0} {\rightarrow} X {\rightarrow} X^{\geqslant 1} \rightarrow \Sigma X^{\leqslant 0},$$
where $X^{\leqslant 0} \in \mathcal{D}^{\leqslant 0}$ and $X^{\geqslant 1} \in \mathcal{D}^{\geqslant 1}$.
\end{itemize}
The subcategories $\mathcal{D}^{\leqslant 0}$ and $\mathcal{D}^{\geqslant 1}$ are called the \textbf{aisle} and \textbf{coaisle} of $\mathbbm{t}$ respectively.
\end{definition}

Let $\mathbbm{t}=(\mathcal{D}^{\leqslant 0},\mathcal{D}^{\geqslant 1})$ and $\mathbbm{t_1}=(\mathcal{D}_{1}^{\leqslant 0},\mathcal{D}_{1}^{\geqslant 1})$ be t-structures on $\mathcal{T}$. The following definitions and notation we will be used later. 
\begin{itemize}
\item For any $n\in \mathbbm{Z}$, let $\mathcal{D}^{\leqslant n}:=\Sigma^{-n} \mathcal{D}^{\leqslant 0}$, $\mathcal{D}^{\geqslant n+1}:=\Sigma^{-n} \mathcal{D}^{\geqslant 1}$ and $\Sigma^{-n} \mathbbm{t}:=(\mathcal{D}^{\leqslant n},\mathcal{D}^{\geqslant n+1})$ which is also a t-structure on $\mathcal{T}$ \cite[Remark 10.1.2]{kashiwara1990sheaves}.
\item $\mathcal{H}_{\mathbbm{t}}:=\mathcal{D}^{\leqslant 0} \cap \mathcal{D}^{\geqslant 0}$ is called the heart of $\mathbbm{t}$. It is an abelian category (see \cite[Proposition 10.1.11]{beilinson1982perverse}). There is a \textbf{cohomological
functor} $H^{0}_{\mathbbm{t}}:\mathcal{T}\longrightarrow \mathcal{H}_{\mathbbm{t}}$ (i.e. a functor sending distinguished triangles in $\mathcal{T}$ to long exact sequences in $\mathcal{H}_{\mathbbm{t}}$) defined by:
$$H^{0}_{\mathbbm{t}}(X):=\tau_{\mathbbm{t}}^{\leqslant 0} \tau_{\mathbbm{t}}^{\geqslant 0}(X)\cong \tau_{\mathbbm{t}}^{\geqslant 0} \tau_{\mathbbm{t}}^{\leqslant 0}(X)\;\text{for all $X\in \mathcal{T}$},$$
where $\tau_{\mathbbm{t}}^{\leqslant 0}$ and $\tau_{\mathbbm{t}}^{\geqslant 0}$ are the \textbf{truncation functors} (i.e. the left and right adjoint functor respectively of the inclusions of $\mathcal{D}^{\leqslant 0}$ and $\mathcal{D}^{\geqslant 0}$ in $\mathcal{T}$). In the same way, one can also define functors $\tau_{\mathbbm{t}}^{\leqslant n}$, $\tau_{\mathbbm{t}}^{\geqslant n}$ and $H^{n}_{\mathbbm{t}}:=\tau_{\mathbbm{t}}^{\leqslant 0}\tau_{\mathbbm{t}}^{\geqslant 0} \Sigma^{n}\cong\Sigma^{n}\tau_{\mathbbm{t}}^{\leqslant n}\tau_{\mathbbm{t}}^{\geqslant n}$ \cite{beilinson1982perverse}.
\item Let $\mathbbm{t}^{+}:=\mathop{\cup}\limits_{n\in \mathbbm{Z}} \mathcal{D}^{\geqslant n}$, $\mathbbm{t}^{-}:=\mathop{\cup}\limits_{n\in \mathbbm{Z}} \mathcal{D}^{\leqslant n}$ and $\mathbbm{t}^{b}:=\mathbbm{t}^{+}\cap \mathbbm{t}^{-}$. The t-structure $\mathbbm{t}$ is called \textbf{bounded below \textnormal{(resp.} bounded above, bounded}) if $\mathbbm{t}^{+}=\mathcal{T}$ (resp. $\mathbbm{t}^{-}=\mathcal{T}$, $\mathbbm{t}^{b}=\mathcal{T}$).
\item $\mathbbm{t}$ and $\mathbbm{t}_{1}$ are called \textbf{equivalent} if there exist $m \leqslant n \in \mathbbm{Z}$ such that $\mathcal{D}^{\leqslant m} \subseteq \mathcal{D}_{1}^{\leqslant 0} \subseteq \mathcal{D}^{\leqslant n}$ (if and only if $\mathcal{D}^{\geqslant m} \subseteq \mathcal{D}_{1}^{\geqslant 0} \subseteq \mathcal{D}^{\geqslant n}$ see \cite[Lemma 4.1]{chen2022extensions}). It is clear that $\mathbbm{t}$ is equivalent to $\Sigma^{n}\mathbbm{t}$ for any $n\in \mathbbm{Z}$.
\end{itemize}

\begin{example}
Let $\mathcal{A}$ be an abelian category. We consider the \textbf{standard t-structure} $\mathbbm{t}_{\mathcal{A}}:=(\mathcal{D}_{\mathcal{A}}^{\leqslant 0}, \mathcal{D}_{\mathcal{A}}^{\geqslant 1})$ on its derived category $\Der^{\ast}(\mathcal{A})$ as follows where $\ast \in \{\emptyset,+,-,b\}$:
$$\mathcal{D}_{\mathcal{A}}^{\leqslant 0}:=\{X \in \Der^{\ast}(\mathcal{A})\;|\; H^{i}(X)=0, \forall i \geqslant 1\}\;,\;\mathcal{D}_{\mathcal{A}}^{\geqslant 1}:=\{X \in \Der^{\ast}(\mathcal{A})\;|\; H^{i}(X)=0, \forall i \leqslant 0\}.$$
In this case, the heart of $\mathbbm{t}_{\mathcal{A}}$ is equivalent to $\mathcal{A}$ \cite{kashiwara1990sheaves}. When $\ast$ is $+$ (resp. $-$, $b$), the standard t-structure $\mathbbm{t}_{\mathcal{A}}$ is bounded below (resp. bounded above, bounded). 

In the following, if $\mathcal{A}=\mods\-R$ or $\mathcal{A}=\Mod\-R$ for a coherent ring $R$, the standard t-structure is denoted by $\mathbbm{t}_{R}$. 

\end{example}

We now turn to the description of objects in the aisle and coaisle of a t-structure.
\begin{lemma}\emph{(\cite[Proposition 10.1.6]{kashiwara1990sheaves})} \label{trancation}
Let $\mathbbm{t}=(\mathcal{D}^{\leqslant 0},\mathcal{D}^{\geqslant 1})$ be a t-structure on $\mathcal{T}$.
     \begin{itemize}
         \item[(1)] If $X\in \mathcal{D}^{\leqslant n}$ \emph{(}resp. $X\in \mathcal{D}^{\geqslant n}$\emph{)}, then $\tau_{\mathbbm{t}}^{\leqslant n}X\cong X$ \emph{(}resp. $\tau_{\mathbbm{t}}^{\geqslant n}X\cong X$\emph{)}.
         \item[(2)] Let $X\in \mathcal{T}$. Then $X\in \mathcal{D}^{\leqslant n}$ \emph{(}resp. $X\in \mathcal{D}^{\geqslant n}$\emph{)} if and only if $\tau_{\mathbbm{t}}^{\geqslant n+1}X=0$ \emph{(}resp. $\tau_{\mathbbm{t}}^{\leqslant n-1}X=0$\emph{)}.
     \end{itemize}
\end{lemma}
\begin{lemma}\label{when it's 0}
Let $\mathbbm{t}$ be a bounded t-structure on $\mathcal{T}$ with heart $\mathcal{H}_{t}$ and $n\in \mathbbm{Z}$. Then 
\begin{enumerate}
\item $X=0$ if and only if $H^{i}_{\mathbbm{t}}(X)=0$ for any $i\in \mathbbm{Z}$.
\item $X \in \mathcal{D}^{\leqslant n}$ \emph{(}resp. $X \in \mathcal{D}^{\geqslant n}$\emph{)} if and only if $H^{i}_{\mathbbm{t}}(X)=0$ \emph{(}resp. $H^{i}_{\mathbbm{t}}(X)=0$\emph{)} for any $i\geqslant n+1$ \emph{(}resp. $i\leqslant n-1$\emph{)}.

\end{enumerate}
Thus $X\in \mathcal{H}_{\mathbbm{t}}$ if and only if $H^{0}_{\mathbbm{t}}(X)\cong X$.
\end{lemma}
\begin{proof}
By \cite[Lemma 2.4]{chen2022extensions}, (1) holds. Here we only prove the first statement in (2). Let $X \in \mathcal{D}^{\leqslant n}$ and $i\geqslant n+1$. Then $\Sigma^{i}X\in \mathcal{D}^{\leqslant -1}$. Hence 
$$H^{i}_{\mathbbm{t}}(X)=H_{\mathbbm{t}}^{0}(\Sigma^{i} X)= 0.$$
For the other implication, we suppose that $X\in \mathcal{T}$ is satisfying $H^{i}_{\mathbbm{t}}(X)=0$. There is a distinguished triangle:
$$X^{\leqslant n}\longrightarrow X \longrightarrow X^{\geqslant n+1} \longrightarrow \Sigma X^{\leqslant n}.$$
By taking the cohomology functor we get the exact sequence in $\mathcal{H}_{\mathbbm{t}}$
$$H^{i}_{\mathbbm{t}}(X^{\leqslant n}) \longrightarrow H^{i}_{\mathbbm{t}}(X) \longrightarrow H^{i}_{\mathbbm{t}}(X^{\geqslant n+1})\longrightarrow H^{i+1}_{\mathbbm{t}}(X^{\leqslant n}).$$
Lemma \ref{trancation} implies that $H^{i}_{\mathbbm{t}}(X^{\leqslant n})=0$. Hence $H^{i}_{\mathbbm{t}}(X)=H^{i}_{\mathbbm{t}}(X^{\geqslant n+1})$. It follows by assumption that $H^{i}_{\mathbbm{t}}(X^{\geqslant n+1})=0$. We also know $H^{i}_{\mathbbm{t}}(X^{\geqslant n+1})=0$ for any $i\leqslant n$ by Lemma \ref{trancation}. Hence $H^{i}_{\mathbbm{t}}(X^{\geqslant n+1})=0$ for all $i\in \mathbbm{Z}$, and $X^{\geqslant n+1}=0$ by (1), which implies $X\cong X^{\leqslant n} \in \mathcal{D}^{\leqslant n}$.
\end{proof}

\subsection{Monoidal t-structures}\label{monoidal_t_str}Until the end of this section, we fix a non-zero monoidal triangulated category $(\mathcal{T},\otimes,\Sigma,\mathbbm{1})$.

\begin{definition}\label{definition of monoidal t-structure}
A bounded t-structure $\mathbbm{t}=(\mathcal{D}^{\leqslant 0},\mathcal{D}^{\geqslant 1})$ on $\mathcal{T}$ is called a \textbf{monoidal t-structure} if there exists $n \in \mathbbm{Z}$ such that 
\begin{enumerate}
\item $\mathcal{D}^{\leqslant 0} \otimes \mathcal{D}^{\leqslant n} \subseteq \mathcal{D}^{\leqslant 0}$;
\item $\mathcal{D}^{\geqslant 0} \otimes \mathcal{D}^{\geqslant n} \subseteq \mathcal{D}^{\geqslant 0}$.
\end{enumerate}
The set of integers $n$ satisfying conditions (1) and (2) is called the \textbf{deviation} of $\mathbbm{t}$ and is denoted by $\text{dev}(\mathbbm{t})$. 
\end{definition}

Our definition of monoidal t-structures is different from, but motivated by Zhang 
and Zhou's definition of monoidal triangulated t-structures (mtt-structures for short) in \cite{zhang2022frobenius}. We later discovered that if $0 \in \text{dev}(\mathbbm{t})$ our definition is equal to the definition of compatible bounded t-structures in \cite{biglari2007kunneth}.

\begin{lemma}\label{shift yy}
Let $\mathbbm{t}$ be a monoidal t-structure on $\mathcal{T}$. For any $k\in \mathbbm{Z}$, $\Sigma^{-k}\mathbbm{t}$ is also a monoidal t-structure. Moreover, if $n \in \textnormal{\text{dev}}(\mathbbm{t})$, then $n-k \in \textnormal{\text{dev}}(\Sigma^{-k}\mathbbm{t})$ for any $k\in \mathbbm{Z}$.
\end{lemma}
\begin{proof}
Since $\mathbbm{t}$ is a monoidal t-structure, there exist $n \in \text{dev}(\mathbbm{t})$ such that 
$$\mathcal{D}^{\leqslant 0} \otimes \mathcal{D}^{\leqslant n} \subseteq \mathcal{D}^{\leqslant 0} \;\text{and}\; \mathcal{D}^{\geqslant 0} \otimes \mathcal{D}^{\geqslant n} \subseteq \mathcal{D}^{\geqslant 0}.$$
Hence for any $k\in \mathbbm{Z}$ 
$$\mathcal{D}^{\leqslant k} \otimes \mathcal{D}^{\leqslant k+n-k} \subseteq \mathcal{D}^{\leqslant k} \;\text{and}\; \mathcal{D}^{\geqslant k} \otimes \mathcal{D}^{\geqslant k+n-k} \subseteq \mathcal{D}^{\geqslant k},$$
which means that $\Sigma^{-k}\mathbbm{t}=(\mathcal{D}^{\leqslant k},\mathcal{D}^{\geqslant k+1})$ is also a monoidal t-structure with $n-k \in \text{dev}(\Sigma^{-k}\mathbbm{t})$.
\end{proof}

Due to Lemma \ref{shift yy}, even if the deviation of a monoidal t-structure $\mathbbm{t}$ may not contain $0$, we can always find an integer $k$ such that $0 \in \text{dev}(\Sigma^{-k}\mathbbm{t})$. Hence in the following we will always assume that $0 \in \text{dev}(\mathbbm{t})$. With the help of this assumption, we are free to apply the following K{\"{u}}nneth formula in a monoidal triangulated category.

\begin{theorem}\textnormal{(\cite[Theorem 4.1]{biglari2007kunneth})}\label{K.lemma}
Let $\mathcal{T}$ be a monoidal triangulated category and $\mathbbm{t}$ be a monoidal t-structure on $\mathcal{T}$ with $0 \in \textnormal{\text{dev}}(\mathbbm{t})$. Suppose that $X, Y\in\mathcal{T}$ and $p,q,n\in \mathbbm{Z}$ with $p+q=n$. Then there is an isomorphism 
\[ \coprod_{p+q=n}H^{p}_{\mathbbm{t}}(X) \otimes H^{q}_{\mathbbm{t}}(Y) \stackrel{\cong}{\longrightarrow} H^{n}_{\mathbbm{t}}(X\otimes Y)\]
which is natural in $X$, $Y\in \mathcal{T}$.
\end{theorem}

\begin{proposition}\label{yy main1}
Let $\mathbbm{t}$ be a monoidal t-structure on $\mathcal{T}$ with $0 \in \textnormal{\text{dev}}(\mathbbm{t})$. Then $\mathcal{H}_{\mathbbm{t}}$ is a monoidal abelian category with a biexact tensor product, whose unit is $H^{0}_{\mathbbm{t}}(\mathbbm{1})$. 

\end{proposition}
\begin{proof}
Let $X \in \mathcal{H}_{\mathbbm{t}}$. According to Theorem \ref{K.lemma}, there is an isomorphism
\[\coprod_{p+q=0} H^{p}_{\mathbbm{t}}(\mathbbm{1}) \otimes H^{q}_{\mathbbm{t}}(X) \stackrel{\cong}\longrightarrow H^{0}_{\mathbbm{t}}(\mathbbm{1}\otimes X).\; \]
By Lemma \ref{when it's 0}, $H^{0}_{\mathbbm{t}}(X)\cong X$ and $H^{q}_{\mathbbm{t}}(X)=0$ for $q\neq 0$. Then
\[H^{0}_{\mathbbm{t}}(\mathbbm{1}) \otimes H^{0}_{\mathbbm{t}}(X) \stackrel{\cong}\longrightarrow H^{0}_{\mathbbm{t}}(\mathbbm{1}\otimes X).\; \]
Similarly, there exists an isomorphism
\[H^{0}_{\mathbbm{t}}(X) \otimes H^{0}_{\mathbbm{t}}(\mathbbm{1}) \stackrel{\cong}\longrightarrow H^{0}_{\mathbbm{t}}( X \otimes \mathbbm{1}).\]
Hence we obtain that
\begin{equation*}
\begin{split}
X \cong H^{0}_{\mathbbm{t}}(X) \cong H^{0}_{\mathbbm{t}}(\mathbbm{1} \otimes X)
\cong H^{0}_{\mathbbm{t}}(\mathbbm{1}) \otimes H^{0}_{\mathbbm{t}}(X) \cong H^{0}_{\mathbbm{t}}(\mathbbm{1}) \otimes X,
\end{split}
\end{equation*}
which is natural in $X$. In the same way, there is a natural isomorphism $X\cong X\otimes H^{0}_{\mathbbm{t}}(\mathbbm{1})$. Since $\mathcal{H}_{\mathbbm{t}}$ is closed under the tensor functor, it is a monoidal abelian category. Moreover, the exactness of the tensor bifunctor is inherited.
\end{proof}

\begin{remark}
For the rest of this paper, when we talk about a monoidal structure on a heart $\mathcal{H}_{\mathbbm{t}}$ where $\mathbbm{t}$ is a monoidal t-structure with $0\in \text{dev}(\mathbbm{t})$, this monoidal structure is inherited from the monoidal structure on $\mathcal{T}$. In addition, $H^{0}_{\mathbbm{t}}(\mathbbm{1}) \neq 0$. Otherwise for any $X \in \mathcal{H}_{\mathbbm{t}}$, $X \cong H^{0}_{\mathbbm{t}}(\mathbbm{1}) \otimes X \cong 0$, hence $\mathcal{H}_{\mathbbm{t}}=0$, which is impossible for a bounded t-structure on $\mathcal{T}$.
\end{remark}

Before proceeding further, we discuss the size of the deviation of a monoidal t-structure $\mathbbm{t}$. Since the unit in the heart of a monoidal t-structure is non-zero, there is not much choice for the deviation. 

\begin{proposition}\label{unique dev}
Let $\mathbbm{t}$ be a monoidal t-structure on $\mathcal{T}$ with $0 \in \textnormal{\text{dev}}(\mathbbm{t})$. Then $\textnormal{\text{dev}}(\mathbbm{t})=\{0\}$.
\end{proposition}
\begin{proof}
Suppose that $a \in \text{dev}(\mathbbm{t})$ and $a\neq 0$. If $a>0$,  then $\mathcal{D}^{\leqslant 0} \otimes \mathcal{D}^{\leqslant a} \subseteq \mathcal{D}^{\leqslant 0}$. Note that $H^{0}_{\mathbbm{t}}(\mathbbm{1}) \in \mathcal{H}_{\mathbbm{t}}\subseteq\mathcal{D}^{\leqslant 0} \subseteq \mathcal{D}^{\leqslant a}$. Since $H^{0}_{\mathbbm{t}}(\mathbbm{1})$ is the unit of $\mathcal{H}_{\mathbbm{t}}$,   
$$H^{0}_{\mathbbm{t}}(\mathbbm{1}) \otimes \Sigma^{-a} H^{0}_{\mathbbm{t}}(\mathbbm{1})\cong \Sigma^{-a}(H^{0}_{\mathbbm{t}}(\mathbbm{1}) \otimes H^{0}_{\mathbbm{t}}(\mathbbm{1}))\cong \Sigma^{-a} H^{0}_{\mathbbm{t}}(\mathbbm{1}) \neq 0.$$ 
However, according to Lemma \ref{when it's 0}, $\Sigma^{-a} H^{0}_{\mathbbm{t}}(\mathbbm{1})$ is not contained in $\mathcal{D}^{\leqslant 0}$, which is a contradiction. The case $a<0$ can be dealt with analogously.
\end{proof}

Next we discuss some examples related to monoidal t-structures. 

\begin{example}\label{key example}
Let $\mathcal{A}$ be a monoidal abelian category with biexact tensor product. We claim that the standard t-structure $\mathbbm{t}_{\mathcal{A}}$ is a monoidal t-structure on $\Der^{b}(\mathcal{A})$. Due to \cite[Lemma 3.4]{biglari2007kunneth}, it is sufficient to prove that $X^{\bullet} \widetilde{\otimes} Y^{\bullet}\in \mathcal{H}_{\mathbbm{t}_{\mathcal{A}}}$ for any $X^{\bullet},Y^{\bullet}\in\mathcal{H}_{\mathbbm{t}_{\mathcal{A}}}$. Provided that $X^{\bullet},Y^{\bullet}\in \mathcal{H}_{\mathbbm{t}_{\mathcal{A}}}$, then $X^{\bullet}\cong H^{0}(X^{\bullet})$ and $Y^{\bullet}\cong H^{0}(Y^{\bullet})$ by Lemma \ref{when it's 0}.

Therefore, 
$$X^{\bullet} \widetilde{\otimes} Y^{\bullet} \cong H^{0}(X^{\bullet}) \otimes H^{0}(Y^{\bullet}) \in \mathcal{H}_{\mathbbm{t}_{\mathcal{A}}},$$ 
which implies that $\mathbbm{t}_{\mathcal{A}}$ is a monoidal t-structure. Moreover, $\mathcal{H}_{\mathbbm{t}_{\mathcal{A}}}$ is monoidal abelian equivalent to $\mathcal{A}$.
\end{example}

The following example shows that there exists a bounded t-structure on a monoidal triangulated category which is not a monoidal t-structure.

\begin{example}
The ring $\mathbbm{Z}$ is hereditary since it is a PID. Hence $\textbf{K}^{b}(\text{proj}\-\mathbbm{Z})$ is triangulated equivalent to $\Der^{b}(\mods\-\mathbbm{Z})$. There is a bounded t-structure on $\textbf{K}^{b}(\text{proj}\-\mathbbm{Z})$, which is the image of the standard t-structure $\mathbbm{t}_{\mathbbm{Z}}$ under this equivalence. Since $(\mods\-\mathbbm{Z}, \otimes_{\mathbbm{Z}}, \mathbbm{Z})$ is a monoidal abelian category, $(\textbf{K}^{b}(\text{mod}\-\mathbbm{Z}), \widetilde{\otimes}, \mathbbm{Z}^{\bullet})$ is a monoidal triangulated category, where $\mathbbm{Z}^{\bullet}$ is the stalk complex with $\mathbbm{Z}$ concentrated in degree $0$. As $\textbf{K}^{b}(\text{proj}\-\mathbbm{Z})$ is closed under the tensor product, then $(\textbf{K}^{b}(\text{proj}\-\mathbbm{Z}), \widetilde{\otimes}, \mathbbm{Z}^{\bullet})$ is also a monoidal triangulated category. Consider the $\mathbbm{Z}$-module $\mathbbm{Z}/6\mathbbm{Z}$ and the complex $X^{\bullet}$ in $\textbf{K}^{b}(\text{proj}\-\mathbbm{Z})$ where
$$X^{\bullet}:\cdots\rightarrow 0\rightarrow \mathbbm{Z} \stackrel{2\cdot}{\longrightarrow}\mathbbm{Z} \rightarrow 0 \rightarrow \cdots$$
and the non-zero terms are only concentrated in degree $-1$ and $0$. By definition both $\mathbbm{Z}/6\mathbbm{Z}$ (as a stalk complex in degree 0) and $X^{\bullet}$ are contained in $\mathcal{H}_{\mathbbm{t}_{\mathbbm{Z}}}$. However, $$(\mathbbm{Z}/6\mathbbm{Z})\widetilde{\otimes} X^{\bullet}\cong \cdots\rightarrow 0\rightarrow \mathbbm{Z}/6\mathbbm{Z} \stackrel{2\cdot}{\longrightarrow}\mathbbm{Z}/6\mathbbm{Z} \rightarrow 0 \rightarrow \cdots$$
is not in $\mathcal{H}_{\mathbbm{t}_{\mathbbm{Z}}}$. Indeed, 
$$H^{-1}_{\mathbbm{t_{\mathbbm{Z}}}}((\mathbbm{Z}/6\mathbbm{Z})\widetilde{\otimes} X^{\bullet})]\cong \mathbbm{Z}/2\mathbbm{Z}\cong H^{0}_{\mathbbm{t_{\mathbbm{Z}}}}((\mathbbm{Z}/6\mathbbm{Z})\widetilde{\otimes} X^{\bullet}).$$ 
Hence $\mathbbm{t}_{\mathbbm{Z}}$ is not a monoidal t-structure on $\textbf{K}^{b}(\text{proj}\-\mathbbm{Z})$. 
\end{example}

\subsection{Tensor reduced monoidal t-structures}\label{Tensor reduced monoidal t-structures}
This subsection is about a special class of monoidal categories called tensor reduced. In this setup, two equivalent monoidal t-structures are equal. In Section \ref{monoidal_t_str}, we observed that $H^{0}_{\mathbbm{t}}(\mathbbm{1})$ is non-zero for a monoidal t-structure $\mathbbm{t}$ on $\mathcal{T}$. In this subsection, we will obtain that $H^{0}_{\mathbbm{t}}(\mathbbm{1})$ is isomorphic to the unit $\mathbbm{1}$ when  $\mathcal{T}$ is a tensor reduced monoidal triangulated category.

\begin{definition}\label{reduced monoidal}
A monoidal additive category $\mathcal{C}$ is called \textbf{tensor reduced} if for any object $X \in \mathcal{C}$, $X\otimes X=0$ if and only if $X=0$. 

Let $\mathcal{T}$ be a monoidal triangulated category.
\begin{enumerate}
\item $\mathcal{T}$ is called \textbf{tensor reduced} if $\mathcal{T}$ is tensor reduced as a monoidal additive category.
\item A monoidal t-structure $\mathbbm{t}$ on $\mathcal{T}$ with $0 \in \text{dev}(\mathbbm{t})$ is called \textbf{tensor reduced}, if $\mathcal{H}_{\mathbbm{t}}$ is tensor reduced as a monoidal additive category.

\end{enumerate}
\end{definition}

Now we are going to discuss some properties of tensor reduced monoidal t-structures.

\begin{lemma}\label{homology group lemma}
Let $\mathbbm{t}$ be a tensor reduced monoidal t-structure on $\mathcal{T}$. Suppose that $X\in \mathcal{D}^{\leqslant m}$  with $H^{m}_{\mathbbm{t}}(X)\neq 0$ and $Y\in \mathcal{D}^{\geqslant n}$ with $H^{n}_{\mathbbm{t}}(Y)\neq 0$ where $m,n \in \mathbbm{Z}$. Then $X \otimes X \in \mathcal{D}^{\leqslant 2m}$ with $H^{2m}_{\mathbbm{t}}(X \otimes X)\neq 0$ and $Y \otimes Y \in \mathcal{D}^{\geqslant 2n}$ with $H^{2n}_{\mathbbm{t}}(Y \otimes Y)\neq 0$.
\end{lemma}
\begin{proof}
We only show the case when $X \in \mathcal{D}^{\leqslant m}$, as the proof for $Y \in \mathcal{D}^{\geqslant n}$ follows in a similar way. Since $\mathbbm{t}$ is a monoidal t-structure with $0 \in \text{dev}(\mathbbm{t})$, 
$$X \otimes X \in \mathcal{D}^{\leqslant m} \otimes \mathcal{D}^{\leqslant m} \subseteq \mathcal{D}^{\leqslant 2m}.$$
By using Theorem \ref{K.lemma}, 
\[H^{2m}_{\mathbbm{t}}(X\otimes X) \cong \coprod_{p+q=2m}H^{p}_{\mathbbm{t}}(X) \otimes H^{q}_{\mathbbm{t}}(X).\]
As $\mathcal{H}_{\mathbbm{t}}$ is tensor reduced and $H^{m}_{\mathbbm{t}}(X)\neq 0$, we obtain that $H^{m}_{\mathbbm{t}}(X) \otimes H^{m}_{\mathbbm{t}}(X)\neq 0$. As $H^{m}_{\mathbbm{t}}(X) \otimes H^{m}_{\mathbbm{t}}(X)$ is a direct summand of $H^{2m}_{\mathbbm{t}}(X\otimes X)$, it follows that $H^{2m}_{\mathbbm{t}}(X\otimes X)\neq 0$.
\end{proof}

\begin{proposition}\label{tensor reduced prop}
Let $\mathbbm{t}$ be a monoidal t-structure on $\mathcal{T}$ with $0 \in \textnormal{\text{dev}}(\mathbbm{t})$. Then $\mathcal{T}$ is tensor reduced if and only if $\mathbbm{t}$ is tensor reduced. Consequently, if $\mathcal{T}$ has a tensor reduced monoidal t-structure, then all monoidal t-structures with $0$ contained in the deviations are tensor reduced.
\end{proposition}
\begin{proof}
If $\mathcal{T}$ is tensor reduced, then it is clear that $\mathbbm{t}$ is tensor reduced. For the other implication, let $X$ be a non-zero object in $\mathcal{T}$. Since $\mathbbm{t}$ is a bounded t-structure, there exists an integer $m\in \mathbbm{Z}$ such that $X\in \mathcal{D}^{\leqslant m}$ with $H^{m}_{\mathbbm{t}}(X)\neq 0$. Due to Lemma \ref{homology group lemma}, we obtain $H^{2m}_{\mathbbm{t}}(X \otimes X)\neq 0$ and therefore $X\otimes X\neq 0$. So $\mathcal{T}$ is tensor reduced.   
\end{proof}

\begin{example}
Let $H$ be a finite-dimensional Hopf algebra. According to Example \ref{key example}, the standard t-structure $\mathbbm{t}_{H}$ on $\Der^{b}(\mods\-H)$ is a monoidal t-structure whose heart is monoidal abelian equivalent to $\mods\-H$.  Moreover, the tensor product on $\mods\-H$ is $\otimes_{\mathbbm{k}}$ by Example \ref{hopf algebra}, which implies that $\mods\-H$ is tensor reduced. Then Proposition \ref{tensor reduced prop} implies that $\Der^{b}(\mods\-H)$ is also tensor reduced.   

\end{example}

\begin{remark}    
Following \cite[Definition 9.14]{miyachi2000derived}, a t-structure is called \textbf{stable} if $\Sigma \mathcal{D}^{\leqslant 0}=\mathcal{D}^{\leqslant 0}$. However the heart of a stable t-structure is $\{0\}$ (see \cite[Lemma 2.5]{chen2022extensions}). Since the monoidal triangulated category $\mathcal{T}$ is non-zero and monoidal t-structures are bounded, a monoidal t-structure must not be stable. Here we recall that a t-structure $\mathbbm{t}$ is not stable if and only if there are integers $a,b$ where $a<b$ such that $\mathcal{D}^{\leqslant a} \subsetneq \mathcal{D}^{\leqslant b}$. 
\end{remark}

\begin{theorem}\label{unique u.t.e}
Let $\mathbbm{t}$ and $\mathbbm{t}_{1}$ be tensor reduced monoidal t-structures on $\mathcal{T}$.  If $\mathbbm{t}$ and $\mathbbm{t}_{1}$ are equivalent, then $\mathbbm{t}=\mathbbm{t}_{1}$.
\end{theorem}
\begin{proof}
Since $\mathbbm{t}$ and $\mathbbm{t}_{1}$ are equivalent, there exist $a\leqslant b\in \mathbbm{Z}$ such that $\mathcal{D}^{\leqslant a} \subseteq \mathcal{D}_{1}^{\leqslant 0} \subseteq \mathcal{D}^{\leqslant b}$. Here we can choose $b$ (resp. $a$) to be the minimal (resp. maximal) integer satisfying this condition, namely, there does not exist an integer $k<b$ (resp. $k>a$) such that $\mathcal{D}^{\leqslant 0}_{1} \subseteq \mathcal{D}^{\leqslant k}$ (resp. $\mathcal{D}^{\leqslant k} \subseteq \mathcal{D}^{\leqslant 0}_{1}$). In this case, there are $X,Y \in \mathcal{D}^{\leqslant 0}_{1}$ such that both $H^{a}_{\mathbbm{t}}(X)$ and $H^{b}_{\mathbbm{t}}(Y)$ are non-zero.

\par If $a=b$, then $\mathbbm{t}_{1}=\Sigma^{-a}\mathbbm{t}$. We can deduce that $a=0$. Indeed, $\mathbbm{t}_{1}=\Sigma^{-a}\mathbbm{t}$ and $\mathbbm{t}_{1}$ is a monoidal t-structure with $0 \in \text{dev}(\mathbbm{t}_{1})$, which implies $\mathcal{D}^{\leqslant a} \otimes \mathcal{D}^{\leqslant a} \subseteq \mathcal{D}^{\leqslant a}$ and $a \in \text{dev}(\mathbbm{t})$. By Proposition \ref{unique dev}, $a=0$. 
\par If $a<b$, then one of the following two cases must occur: $b>0$ or $a<0$.
\begin{enumerate}
    \item Case 1: $b>0$. By Lemma \ref{homology group lemma}, $H^{2b}(Y\otimes Y)\neq 0$. Hence $Y\otimes Y$ is not contained in $\mathcal{D}^{\leqslant 0}_{1}$, which is a contradiction.
    \item Case 2: $a<0$. By Lemma \ref{homology group lemma}, $H^{2a}(X\otimes X)\neq 0$. Hence $X\otimes X$ is not contained in $\mathcal{D}^{\geqslant 0}_{1}$, which is a contradiction.
\end{enumerate}
Hence the only possible case is $a=b=0$, and the proof is complete.
\end{proof}

\begin{corollary}\label{unique u.t.e cor}
Let $F:\mathcal{T}_{1}\rightarrow \mathcal{T}_{2}$ be a monoidal triangulated equivalence between monoidal triangulated categories $\mathcal{T}_{1}$ and $\mathcal{T}_{2}$. Suppose that $\mathbbm{t}_{1}$ and $\mathbbm{t}_{2}$ are monoidal t-structures on $\mathcal{T}_{1}$ and $\mathcal{T}_{2}$ respectively. The $F$-images of aisle and coaisle of $\mathbbm{t}_{1}$ in $\mathcal{T}_{2}$ are denoted by
$$\mathcal{U}=\{X\in \mathcal{T}_{2}\;|\; \exists \Tilde{X} \in  \mathcal{D}_{1}^{\leqslant 0} \text{ such that } F(\Tilde{X}) \cong X\}, $$
$$\mathcal{V}=\{Y\in \mathcal{T}_{2}\;|\; \exists \Tilde{Y} \in  \mathcal{D}_{1}^{\geqslant 1} \text{ such that } F(\Tilde{Y}) \cong Y\}. $$
Then the following statements hold:
\begin{enumerate}
\item $\mathbbm{t}:=(\mathcal{U},\mathcal{V})$ is a monoidal t-structure on $\mathcal{T}_{2}$.
\item Suppose $0\in \textnormal{\text{dev}}(\mathbbm{t}_{1})$ and $\mathbbm{t}_{2}$ is tensor reduced. If $\mathbbm{t}$ and $\mathbbm{t}_{2}$ are equivalent, then the restriction 
$$F|_{\mathcal{H}_{\mathbbm{t}_{1}}}:\mathcal{H}_{\mathbbm{t}_{1}} \rightarrow \mathcal{H}_{\mathbbm{t}_{2}}$$
is a monoidal abelian equivalence.
\end{enumerate}
\end{corollary}
\begin{proof}
As $F$ is a monoidal triangulated equivalence and $\mathbbm{t}_{1}$ is a monoidal t-structure, by definition $\mathbbm{t}$ is also a monoidal t-structure on $\mathcal{T}_{2}$. We are now turning to the proof of (2). 
\par Since the monoidal structures of $\mathcal{H}_{\mathbbm{t}_{1}}$ and $\mathcal{H}_{\mathbbm{t}}$ are inherited from $\mathcal{T}_{1}$ and $\mathcal{T}_{2}$ respectively, the restriction 
$$F|_{\mathcal{H}_{\mathbbm{t}_{1}}}:\mathcal{H}_{\mathbbm{t}_{1}} \rightarrow \mathcal{H}_{\mathbbm{t}}$$
is a monoidal abelian equivalence. By assumption, $0\in \textnormal{\text{dev}}(\mathbbm{t}_{1})$ which implies $0\in \textnormal{\text{dev}}(\mathbbm{t})$. Since $\mathbbm{t}_{2}$ is tensor reduced by assumption, $\mathbbm{t}$ is also tensor reduced by Proposition \ref{tensor reduced prop}.  Now that 
 $\mathbbm{t}$ and $\mathbbm{t}_{2}$ are equivalent, it follows that $\mathbbm{t}=\mathbbm{t}_{2}$ by Theorem \ref{unique u.t.e}. Hence $\mathcal{H}_{\mathbbm{t}}=\mathcal{H}_{\mathbbm{t}_{2}}$, which completes the proof.
\end{proof}

We conclude this section by discussing the heart $\mathcal{H}_{\mathbbm{t}}$ of a tensor reduced monoidal t-structure $\mathbbm{t}$ on $\mathcal{T}$. Furthermore, we notice that rigidity of $\mathcal{T}$ gives rigidity of 
$\mathcal{H}_{\mathbbm{t}}$.
\begin{proposition}\label{ring prop}
Let $\mathbbm{t}$ be a tensor reduced monoidal t-structure on $\mathcal{T}$. Then $H^{i}_{\mathbbm{t}}(\mathbbm{1})=0$ for all $i\neq 0$, that is $\mathbbm{1}\cong H^{0}_{\mathbbm{t}}(\mathbbm{1}) \in \mathcal{H}_{\mathbbm{t}}$.
\end{proposition}
\begin{proof}
Since $\mathbbm{t}$ is bounded, there exist integers $m\geqslant n$ such that $\mathbbm{1}\in \mathcal{D}^{\leqslant m} \cap \mathcal{D}^{\geqslant n}$. Let
$$m:=\text{min}\{a\in \mathbbm{Z}\;|\; H^{i}_{\mathbbm{t}}(\mathbbm{1})=0,\; \forall i>a \},$$
$$n:=\text{max}\{a\in \mathbbm{Z}\;|\; H^{i}_{\mathbbm{t}}(\mathbbm{1})=0,\; \forall i< a \}.$$
The choice of $m$ tells us that $H^{m}_{\mathbbm{t}}(\mathbbm{1})\neq 0$. Moreover, $m\geqslant 0$ because $H^{0}_{\mathbbm{t}}(\mathbbm{1})\neq 0$. If $m>0$, then $2m>m>0$ and therefore $H^{2m}_{\mathbbm{t}}(\mathbbm{1})=0$ which contradicts Lemma \ref{homology group lemma}. Hence $m=0$. Likewise $n=0$. To sum up, $H^{i}_{\mathbbm{t}}(\mathbbm{1})=0$ for any $i\neq 0$, and $\mathbbm{1}\cong H^{0}_{\mathbbm{t}}(\mathbbm{1}) \in \mathcal{H}_{\mathbbm{t}}$ by Lemma \ref{when it's 0}.
\end{proof}

\begin{corollary}\label{dual prop}
Let $\mathbbm{t}$ be a tensor reduced monoidal t-structure on $\mathcal{T}$. 
\begin{enumerate}
\item If $\mathcal{T}$ has left duals, then $(\mathcal{D}^{\leqslant 0})^{\ast} \subseteq \mathcal{D}^{\geqslant 0}$ and $(\mathcal{D}^{\geqslant 0})^{\ast} \subseteq \mathcal{D}^{\leqslant 0}$, which implies $\mathcal{H}_{\mathbbm{t}}$ has left duals. 
\item If $\mathcal{T}$ has right duals, then $^{\ast}(\mathcal{D}^{\leqslant 0})\subseteq \mathcal{D}^{\geqslant 0}$ and $^{\ast}(\mathcal{D}^{\geqslant 0}) \subseteq \mathcal{D}^{\leqslant 0}$, which implies $\mathcal{H}_{\mathbbm{t}}$ has right duals. 
\end{enumerate}
As a consequence, if $\mathcal{T}$ is rigid and $\emph{End}_{\mathcal{A}}(\mathbbm{1})\cong \mathbbm{k}$, then $\mathcal{H}_{\mathbbm{t}}$ is a tensor category.
\end{corollary}
\begin{proof}
We only prove (1), and the proof of (2) is similar. Because $\mathbbm{t}$ is a monoidal t-structure with $0 \in \textnormal{\text{dev}}(\mathbbm{t})$, by definition $\mathcal{D}^{\leqslant 0} \otimes \mathcal{D}^{\leqslant 0} \subseteq \mathcal{D}^{\leqslant 0}$ and $\mathcal{D}^{\geqslant 0} \otimes \mathcal{D}^{\geqslant 0} \subseteq \mathcal{D}^{\geqslant 0}$. Since $\mathbbm{t}$ is tensor reduced,  $\mathbbm{1} \in \mathcal{H}_{\mathbbm{t}}$ by applying Proposition \ref{ring prop}. Due to \cite[Proposition 2.10.8.]{etingof2016tensor}, there are isomorphisms
$$\Hom_{\mathcal{T}}(\mathcal{D}^{\leqslant -1},\mathbbm{1}\otimes {(\mathcal{D}^{\leqslant 0})}^{\ast}) \cong \Hom_{\mathcal{T}}(\mathcal{D}^{\leqslant -1}\otimes \mathcal{D}^{\leqslant 0},\mathbbm{1})=0,$$
$$\Hom_{\mathcal{T}}(({\mathcal{D}^{\geqslant 0}})^{\ast} \otimes \mathbbm{1}, \mathcal{D}^{\geqslant 1}) \cong \Hom_{\mathcal{T}}(\mathbbm{1}, \mathcal{D}^{\geqslant 0}\otimes \mathcal{D}^{\geqslant 1})=0.$$
This yields that  ${(\mathcal{D}^{\leqslant 0})}^{\ast}=\mathbbm{1}\otimes {(\mathcal{D}^{\leqslant 0})}^{\ast} \subseteq \mathcal{D}^{\geqslant 0}$ and $(\mathcal{D}^{\geqslant 0})^{\ast} = ({\mathcal{D}^{\geqslant 0}})^{\ast} \otimes \mathbbm{1}\subseteq \mathcal{D}^{\leqslant 0}$. Hence 
$${\mathcal{H}_{\mathbbm{t}}}^{\ast} \subseteq {(\mathcal{D}^{\leqslant 0})}^{\ast} \cap (\mathcal{D}^{\geqslant 0})^{\ast} \subseteq \mathcal{D}^{\geqslant 0} \cap \mathcal{D}^{\leqslant 0}=\mathcal{H}_{\mathbbm{t}},$$
and we finish the proof.
\end{proof}

A natural question is whether the converse of Corollary \ref{dual prop} holds. More specifically, given a tensor reduced monoidal t-structure $\mathbbm{t}$ on $\mathcal{T}$, if the heart $\mathcal{H}_{\mathbbm{t}}$ is rigid, can we conclude that $\mathcal{T}$ is rigid as well? The general case is unclear to us, but the following shows that in some cases, the rigidity of the heart of a monoidal t-structure implies that the monoidal triangulated category is rigid.

\begin{example}
Suppose $H$ is a finite-dimensional Hopf algebra. We already know that $\mods\-H$ is rigid. In the following, we will see that $\Der^{b}(\mods\-H)$ is rigid as well. Since there is a monoidal dense functor from $\Com^b(\mods\-H)$ to $\Der^{b}(\mods\-H)$, it suffices, by remark \ref{rigid under monoidal}, to show that the category of complexes $\Com^b(\mods\-H)$ (which is also a monoidal category whose unit is the stalk complex $\mathbbm{k}^{\bullet}$ with $\mathbbm{k}$ in degree $0$) is rigid. Indeed, let $X^{\bullet}$ be any complex in $\Com^b(\mods\-H)$ denoted by 
$$X^{\bullet}: 0\rightarrow X^{-n}\xrightarrow{d_X^{-n}} X^{-n+1}\xrightarrow{d_X^{-n+1}}\cdots\xrightarrow{d_X^{-1}} X^0\xrightarrow{d_X^{0}}\cdots\xrightarrow{d_X^{m-2}} X^{m-1}\xrightarrow{d_X^{m-1}} X^m\rightarrow 0$$
where $n$, $m$ are positive integers.
Let $Y^{-i}$ be the left dual of $X^i$ where $-n\leqslant i\leqslant m$. Then there are evaluations and coevaluations denoted by 
$$\epsilon_i:Y^{-i}\otimes_\mathbbm{k} X^i\rightarrow \mathbbm{k},\;\; \eta_i:\mathbbm{k}\rightarrow X^i\otimes_\mathbbm{k} Y^{-i}.$$
We claim that the complex 
$$Y^{\bullet}: 0\rightarrow Y^{-m}\xrightarrow{(-1)^{-m+2}(d_X^{m-1})^* }Y^{-m+1}\xrightarrow{ }\cdots\xrightarrow{}  
\xrightarrow{(-1)^{n}(d_X^{-n+1})^*} Y^{n-1}\xrightarrow{(-1)^{n+1}(d_X^{-n})^* } Y^n\rightarrow 0$$
is the left dual of $X^{\bullet}$. Let $\epsilon:=\mathop{\bigoplus}\limits_{i=-n}^{m}\epsilon_i$ and $\eta:=\mathop{\bigoplus}\limits_{i=-n}^{m}\eta_{i}$. It is direct to see that $\epsilon$ and $\eta$ are chain maps 
$$\epsilon: Y^{\bullet}\widetilde{\otimes} X^{\bullet}\rightarrow \mathbbm{k}^{\bullet},\;\; \eta: \mathbbm{k}^{\bullet}\rightarrow X^{\bullet}\widetilde{\otimes} Y^{\bullet},$$
where $X^{\bullet}\widetilde{\otimes} Y^{\bullet}$ (resp. $Y^{\bullet}\widetilde{\otimes} X^{\bullet}$) denotes the total complex. Thus $Y^{\bullet}$ is the left dual of $X^{\bullet}$ with evaluation $\epsilon$ and coevaluation $\eta$. Likewise, we can construct the right dual of $X^{\bullet}$. 
\end{example}

\section{Locally finite tensor Grothendieck categories are tensor reduced}\label{L.N.T.R.G}
In this section, we introduce the concept of a locally finite tensor Grothendieck category. Our motivation for considering this category stems from the desire to explore certain monoidal abelian categories with biexact tensor product, which are not rigid but contain a finite tensor subcategory. For example, $\Mod\-\mathbbm{k}$ is not a rigid monoidal category having a finite tensor abelian subcategory $\mods\-\mathbbm{k}$. The tensor product on $\Mod\-\mathbbm{k}$ can be constructed using the filtered colimit of objects in $\mods\-\mathbbm{k}$, which is consistent with the tensor product over $\mathbbm{k}$. Inspired by this observation, we will first explain an analogous statement for locally coherent Grothendieck categories.

\subsection{Locally finite tensor Grothendieck categories}
According to \cite[Proposition 1.8.1]{grothendieck1957quelques}, an abelian category $\mathcal{A}$ satisfies \textbf{(Ab.5)} means that if and only if it is cocomplete and for every filtered category $I$ the colimit functor $\Colim:\text{Fun}[I,\mathcal{A}] \rightarrow \mathcal{A}$ is exact. An abelian category $\mathcal{A}$ is called a \textbf{Grothendieck category}, if it satisfies (Ab.5) and has a generator.

\begin{definition}\cite[Section 1.2 and 1.3]{herzog1997the}
Let $\mathcal{A}$ be a Grothendieck category. 
\begin{enumerate}
\item An object $A\in\mathcal{A}$ is called \textbf{finitely presented}, if the presentable functor $\Hom_{\mathcal{A}}(A,-):\mathcal{A}\rightarrow \Ab$ preserves filtered colimits. The full subcategory of $\mathcal{A}$ consisting of finitely presented objects will be denoted by $\fp(\mathcal{A})$.
\item $\mathcal{A}$ is called a \textbf{locally finitely presented Grothendieck category} if $\fp(\mathcal{A})$ is skeletally small and generates $\mathcal{A}$, i.e. any object is a filtered colimit of some finitely presented objects. According to \cite[Theorem 1.6]{herzog1997the}, if in addition, $\fp(\mathcal{A})$ is an abelian category, then $\mathcal{A}$ is called a \textbf{locally coherent Grothendieck category}.
\end{enumerate}
\end{definition}

Now let $\mathcal{A}$ be a locally coherent Grothendieck category.  Since $\fp(\mathcal{A})$ generates $\mathcal{A}$, we explore whether there are some additional structures on $\fp(\mathcal{A})$ inducing the corresponding structures on $\mathcal{A}$. The case we are interested in is that $\fp(\mathcal{A})$ satisfies the following finiteness condition. 

An abelian category $\mathcal{C}$ is said to be \textbf{finite} if the following conditions are satisfied:
\begin{enumerate}

\item every object in $\mathcal{C}$ has finite length;
\item $\mathcal{C}$ has finitely many isomorphism classes of simple objects. 

\end{enumerate}

\begin{remark}
There are some different variants of finiteness. The definition of finite abelian category needs the additional assumption that every simple has a projective cover in some references (for example in \cite[Definition 1.8.6]{etingof2016tensor}). 
\end{remark}

\begin{proposition}\textnormal{(\cite[Section 3.4]{hoha2002embedding})}\label{monoidal structure}
Let $\mathcal{A}$ be a locally coherent Grothendieck category. If $\textnormal{\fp}(\mathcal{A})$ is a monoidal abelian category, then $\mathcal{A}$ is a monoidal abelian category and the tensor product bifunctor commutes with filtered colimits in each variable. If the tensor product on $\textnormal{\fp}(\mathcal{A})$ is biexact, then the induced tensor product on $\mathcal{A}$ is also biexact.  
\end{proposition}

\begin{remark}
The construction of the tensor product in a locally coherent Grothendieck category $\mathcal{A}$ is as follows. According to \cite[Corollary 1.4, p.1648]{crawley1994locally}, every object in $\mathcal{A}$ can be written as a filtered colimit of objects $X_i\in\fp(\mathcal{A})$ indexed by a filtered category $I$ as follows
$$\Colim\limits_{i \in I}X_{i}\cong X.$$
For any objects $A$ and $B$ in $\mathcal{A}$, as $\fp(\mathcal{A})$ is a tensor category, the tensor product in $\mathcal{A}$ is defined to be
$$A\otimes B:= \Colim_{i \in I} \Colim_{j\in J} A_{i} \otimes B_{j},$$
where $I$, $J$ are filtered categories, and \textbf{$\{A_{i}\}_{i\in I}$, $\{B_{j}\}_{j\in J}$} are objects in $\textnormal{\fp}(\mathcal{A})$ such that
$$\Colim\limits_{i \in I}A_{i}\cong A, \;\;\Colim\limits_{j \in J}B_{j}\cong B.$$ 

This tensor product is well-defined, see \cite[Section 3.4]{hoha2002embedding}.
\end{remark}

\begin{definition}\label{L.N.T.G cat}
Let $\mathcal{A}$ be a locally coherent Grothendieck category. If $\fp(\mathcal{A})$ is a finite tensor category, then $\mathcal{A}$ is called a \textbf{locally finite tensor Grothendieck category}. 
\end{definition}

Now we are ready to discuss the monoidal abelian equivalences between locally finite tensor Grothendieck categories. Given locally finite tensor Grothendieck categories $\mathcal{A}$ and $\mathcal{B}$, a functor from $\fp(\mathcal{A})$ to $\fp(\mathcal{B})$ can be extended to a functor from $\mathcal{A}$ to $\mathcal{B}$, as shown in the following statement.

\begin{lemma}\textnormal{(\cite[Lemma 2.7]{krause1997spectrum})}\label{extendsion lemma1}
Let $\mathcal{A}$ and $\mathcal{B}$ be locally finite tensor Grothendieck categories. Any functor $F:\textnormal{\fp}(\mathcal{A}) \rightarrow \textnormal{\fp}(\mathcal{B})$ can be extended uniquely, up to isomorphism, to a functor $\Tilde{F}:\mathcal{A} \rightarrow \mathcal{B}$ which commutes with filtered colimits. If $F$ is an exact monoidal functor, so is $\Tilde{F}$. Moreover, when $F$ is a monoidal abelian equivalence, $\Tilde{F}$ is also a monoidal abelian equivalence.
\end{lemma}
\begin{proof}
Let $F:\textnormal{\fp}(\mathcal{A}) \rightarrow \textnormal{\fp}(\mathcal{B})$ be an exact monoidal functor. According to \cite[Lemma 2.7]{krause1997spectrum}, we only need to prove that $\Tilde{F}$ is a monoidal functor. For any $X,Y \in \mathcal{A}$, there are isomorphisms
\begin{equation*}
\begin{split}
\Tilde{F}(X\otimes Y)&=\Tilde{F}(\Colim_{i\in I} \Colim_{j\in J} X_{i}\otimes Y_{j})\cong \Colim_{i\in I} \Colim_{j\in J} \Tilde{F}(X_{i}\otimes Y_{j}) \\
&=\Colim_{i\in I} \Colim_{j \in J} F(X_{i}\otimes Y_{j})\cong \Colim_{i\in I} \Colim_{j\in J} F(X_{i})\otimes F(Y_{j})\\  
&\cong \Tilde{F}(X)\otimes \Tilde{F}(Y)
\end{split}
\end{equation*} 
Here, each isomorphism is natural in $X$ and $Y$. It follows that $\Tilde{F}$ is also a monoidal functor. In addition, if $F$ is an equivalence, then assume that $G$ is the quasi-inverse of $F$. By the same way $G$ can be extended to a monoidal functor $\Tilde{G}$, and it is clear that $\Tilde{G}$ is the quasi-inverse of $\Tilde{F}$.
\end{proof}

It is part of Morita theory that for finite-dimensional algebras $A$ and $B$, $\mods\- A$ is equivalent to $\mods\- B$ if and only if $\Mod\- A$ is equivalent to $\Mod\- B$. We obtain a similar result in the case of locally finite tensor Grothendieck categories.

\begin{proposition}\label{Morita tensor eq}
Let $\mathcal{A}$ and $\mathcal{B}$ be locally finite tensor Grothendieck categories. Then the following statements are equivalent:
\begin{enumerate}
\item $\textnormal{\fp}(\mathcal{A})$ and $\textnormal{\fp}(\mathcal{B})$ are monoidal abelian equivalent;
\item $\mathcal{A}$ and $\mathcal{B}$ are monoidal abelian equivalent.
\end{enumerate}
\end{proposition}
\begin{proof}
By Lemma \ref{extendsion lemma1}, $(1) \Rightarrow (2)$. For the other implication, let $\Tilde{F}:\mathcal{A} \rightarrow \mathcal{B}$ be a monoidal abelian equivalence with quasi-inverse $\Tilde{G}$. Suppose that $A\in \fp(\mathcal{A})$. The first thing we need to verify is that $\Tilde{F}(A)\in \fp(\mathcal{B})$. Since $\Tilde{F}$ is an equivalence, for any filtered category $I$ and objects $\{B_{i}\}_{i\in I}\in\mathcal{B}$, there exist objects $\{A_{i}\}_{i\in I}\in\mathcal{A}$ such that $\Tilde{F}(A_{i}) \cong B_{i}$ and $\Tilde{F}(\Colim\limits_{i\in I} A_{i}) \cong \Colim\limits_{i\in I} \Tilde{F}(A_{i})$. Then 
\begin{equation*}
\begin{split}
\Hom_{\mathcal{B}}(\Tilde{F}(A),\Colim_{i\in I} B_{i})&\cong \Hom_{\mathcal{B}}(\Tilde{F}(A),\Colim_{i\in I}\Tilde{F}(A_{i}))\\
& \cong \Hom_{\mathcal{B}}(\Tilde{F}(A),\Tilde{F}(\Colim_{i\in I}A_{i}))\\
(\text{by}\; \Tilde{F}\; \text{is an equivalence})\;& \cong \Hom_{\mathcal{A}}(A,\Colim_{i\in I} A_{i})\\
(\text{by}\; A\in \fp(\mathcal{A}))\;& \cong \Colim_{i\in I}\Hom_{\mathcal{A}}(A, A_{i})\\
(\text{by}\; \Tilde{F}\; \text{is an equivalence})\;& \cong \Colim_{i\in I} \Hom_{\mathcal{B}}(\Tilde{F}(A), B_{i}),
\end{split}
\end{equation*}
which means that $\Tilde{F}(A)\in \fp(\mathcal{B})$. Hence $\Tilde{F}(\fp(\mathcal{A})) \subseteq \fp(\mathcal{B})$. Now let $F:=\Tilde{F}|_{\fp(\mathcal{A})}$. Then $F$ is a fully faithful exact monoidal functor from $\fp(\mathcal{A})$ to $\fp(\mathcal{B})$ due to the fact that they are abelian subcategories closed under the tensor product in $\mathcal{A}$ and $\mathcal{B}$ respectively. For the same reason, the claim also holds for the functor $G:=\Tilde{G}|_{\fp(\mathcal{B})}$, thus $G$ is precisely the quasi-inverse of $F$, which implies that $F$ is an equivalence. 
\end{proof}

\begin{example}\label{hopf morita}
Let $H$ be a finite-dimensional Hopf algebra. Since $\fp(\Mod\-H)=\mods\-H$ is a finite tensor category, $\Mod\-H$ is a locally finite tensor Grothendieck category. Suppose $H'$ is another finite-dimensional Hopf algebra. By Proposition \ref{Morita tensor eq}, $\Mod\-H$ and $\Mod\-H'$ are monoidal abelian equivalent if and only if $\mods\-H$ and $\mods\-H'$ are monoidal abelian equivalent. 

\end{example}
\begin{remark}\label{gauge}
    Let $H$ and $H'$ be finite-dimensional Hopf algebras. According to \cite[Definition \uppercase\expandafter{\romannumeral15}.3.3]{kassel1995quantum}, $H$ and $H'$ are said to be \textbf{gauge equivalent}, if there is a gauge transformation $J$ of $H$, such that $H'$ and $H^J$ are isomorphic as bialgebras. By \cite[Theorem 2.2]{ng2008central} $\mods\-H$ and $\mods\-H'$ are monoidal abelian equivalent if and only if $H$ is gauge equivalent to $H'$.

\end{remark}

\subsection{Tensor reducedness}
In this subsection, we mainly discuss tensor reducedness of tensor categories and locally finite tensor Grothendieck categories, which leads to the main results in Section \ref{Morita/Rickard eq}.

Recall that for an abelian category $\mathcal{A}$, a subcategory $\mathcal{S}$ of $\mathcal{A}$ is called a \textbf{Serre subcategory} if $\mathcal{S}$ is closed under taking subobjects, quotients and extensions. For a monoidal category $\mathcal{C}$, a subcategory $\mathcal{I}$ is called a \textbf{left} (resp. \textbf{right}) \textbf{tensor ideal} if $Y \otimes X \in \mathcal{I}$ (resp. $X \otimes Y \in \mathcal{I}$) for any $X \in \mathcal{I}, Y\in \mathcal{C}$. If $\mathcal{I}$ is both a left and right tensor ideal, then it is called (\textbf{two-sided}) \textbf{tensor ideal}. Then a Serre subcategory $\mathcal{S}$ in a monoidal abelian category is called a \textbf{left} (resp. \textbf{right}, \textbf{two-sided}) \textbf{Serre tensor ideal} if it is a left (resp. right, two-sided) tensor ideal.

We first present examples of Serre tensor ideals. Let $\mathcal{A}$ be a tensor category and $X$ be a non-zero object in $\mathcal{A}$. The \textbf{left annihilator} \textnormal{(}resp. \textbf{right annihilator}\textnormal{)} of $X$ is denoted by
$$\textnormal{\text{l.Ann}}(X):=\{A\in \mathcal{A}\;|\; A\otimes X=0\}\;\;(\text{resp.} \;\;\textnormal{\text{r.Ann}}(X):=\{A\in \mathcal{A}\;|\; X\otimes A=0\}),$$
which is a full subcategory of $\mathcal{A}$.

\begin{lemma}\label{tensor ideal lemma}
Let $\mathcal{A}$ be a tensor category and $X$ be a non-zero object in $\mathcal{A}$.  Then $\textnormal{\text{l.Ann}}(X)$ \textnormal{(}resp. $\textnormal{\text{r.Ann}}(X)$\textnormal{)} is a left \textnormal{(}resp. right\textnormal{)} Serre tensor ideal. 

\end{lemma}
\begin{proof}
We only prove that $\textnormal{\text{l.Ann}}(X)$ is a left Serre tensor ideal. Firstly, we need to check that it is a Serre subcategory. Let 
$$0\rightarrow Y_{0} \rightarrowtail Y_{1} \twoheadrightarrow Y_{2} \rightarrow 0$$
be a short exact sequence in $\mathcal{A}$. By applying the functor $- \otimes X$, there is an exact sequence 
$$0\rightarrow Y_{0}\otimes X \rightarrowtail Y_{1}\otimes X \twoheadrightarrow Y_{2}\otimes X \rightarrow 0.$$
Hence $Y_{1} \otimes X$ is zero if and only if both $Y_{0} \otimes X$ and $Y_{2} \otimes X$ are zero, which implies that $\text{l.Ann}(X)$ is a Serre subcategory. Next, for any $Y\in \mathcal{A}$ and $Z\in \text{l.Ann}(X)$, we can see that $(Y\otimes Z)\otimes X \cong  Y\otimes (Z \otimes X)=0$. Thus $Y\otimes Z\in \text{l.Ann}(X)$.
\end{proof}

Inspired by \textnormal{\cite[Proposition 4.11]{zuo2024quotient}}, we obtain the following statement. 

\begin{proposition}\label{trivial ideal}
Let $\mathcal{A}$ be a finite tensor category. Then all the left and right Serre tensor ideals of $\mathcal{A}$ are trivial. As a consequence, $\mathcal{A}$ is tensor reduced. 

\end{proposition}
\begin{proof}
Let $\mathcal{S}$ be a non-zero left Serre tensor ideal of $\mathcal{A}$. Then there exists a non-zero object $X$ in $\mathcal{S}$. We consider the object $X^{\ast} \otimes X$, which is contained in $\mathcal{S}$. Since $X$ is non-zero and the following composition is the identity morphism of $X$
$$X \stackrel{\coev_{X} \otimes \text{id}_{X}}{\xrightarrow{\hspace*{1.5cm}}} X\otimes X^{\ast} \otimes X  \stackrel{\text{id}_{X}\otimes \ev_{X}}{\xrightarrow{\hspace*{1.5cm}}} X,$$
$X^{\ast} \otimes X$ is also non-zero. The assumption that $\mathcal{A}$ is a finite tensor category implies that $\mathbbm{1}$ is a simple object in $\mathcal{A}$ (\cite[Theorem 4.3.8]{etingof2016tensor}). Hence the evaluation $\text{ev}_{X}:X^{\ast} \otimes X\rightarrow \mathbbm{1}$ is an epimorphism. Note that $\mathcal{S}$ is a Serre subcategory, therefore $\mathbbm{1}$ is an object in $\mathcal{S}$. Thus $\mathcal{S}=\mathcal{A}$. 
\par Let $X$ be a non-zero object in $\mathcal{A}$. According to Lemma \ref{tensor ideal lemma}, $\text{l.Ann}(X)$ is a left Serre tensor ideal of $\mathcal{A}$. Consequently, $\text{l.Ann}(X)$ is either 0 or $\mathcal{A}$. $X$ being non-zero implies that $\text{l.Ann}(X)= 0$, which shows that $X\otimes X \neq 0$.  
\end{proof}

With the necessary groundwork completed, now we can finish the proof of the main result on tensor reducedness. 

\begin{theorem}\label{0tensor lemma}
Let $\mathcal{A}$ be a locally finite tensor Grothendieck category. Then $\mathcal{A}$ is tensor reduced.   
\end{theorem}
\begin{proof}
Let $A$ be a non-zero object in $\mathcal{A}$. Since $\fp(\mathcal{A})$ generates $\mathcal{A}$, there exists an object $X$ in $\fp(\mathcal{A})$ such that $\Hom_{\mathcal{A}}(X,A)\neq 0$. Let $f\in\Hom_{\mathcal{A}}(X,A)$ be a non-zero morphism, so $\Im(f)\neq 0$. $\fp(\mathcal{A})$ is closed under quotients as it is an abelian subcategory of $\mathcal{A}$. Hence $\Im(f)$ is finitely presented as well.
$$\begin{tikzcd}
X && A \\
& {\Im(f)}
\arrow["{f}", from=1-1, to=1-3]
\arrow["{}"', two heads, from=1-1, to=2-2]
\arrow["{}"', tail, from=2-2, to=1-3]
\end{tikzcd}$$
According to Proposition \ref{trivial ideal}, $\fp(\mathcal{A})$ is tensor reduced, hence $\Im(f)\otimes \Im(f)\neq 0$. Then there is a non-zero monomorphism 
$$f\otimes f: \Im(f)\otimes \Im(f) \rightarrowtail A \otimes A,$$
which implies $A\otimes A \neq 0$.
\end{proof}

We end this section by presenting two examples. The first one illustrates that there is a tensor reduced monoidal abelian category that is not a tensor category. The second one is an example of a monoidal abelian category which is not tensor reduced.

\begin{example}\label{rep exacmple}
Let $Q=(Q_{0},Q_{1})$ be a finite acyclic quiver and $|Q_{0}|\geqslant 2$. Consider the category $\text{rep}_{\mathbbm{k}}(Q)$ of finite-dimensional representations of the quiver of type $Q$. According to \cite[Chapter 3, Definition 1.1]{assem2006elements}, such representations are defined by the following data:
\begin{enumerate}
\item to each point $a$ in $Q_{0}$ is associated a finite-dimensional $\mathbbm{k}$-vector space $V_{a}$;
\item to each arrow $\alpha:a\rightarrow b$ in $Q_{1}$ is associated a $\mathbbm{k}$-linear map $\phi_{\alpha}:V_{a}\rightarrow V_{b}$.
\end{enumerate}
Such a representation is denoted by $(V_{a},\phi_{\alpha})_{a\in Q_{0},\alpha\in Q_{1}}$ or simply $(V_{a},\phi_{\alpha})$. The tensor product of two representations is defined by the formula
$$(V_{a},\phi_{\alpha}) \otimes (W_{a},\varphi_{\alpha}):=(V_{a} \otimes_{\k} W_{a}, \phi_{\alpha}\otimes_{\k} \varphi_{\alpha}).$$
By construction, $\text{rep}_{\mathbbm{k}}(Q)$ is a monoidal abelian category with biexact tensor product. The category $\text{rep}_{\mathbbm{k}}(Q)$ is not rigid, as its unit object $(\mathbbm{k}, \text{id})$ is not simple by \cite[Theorem 4.3.8]{etingof2016tensor}. Note that such a finite-dimensional representation $(V_{a},\phi_{\alpha})$ 
 is the zero object if and only if $V_{a}=0$ for all $a\in Q_{0}$. Hence, $\text{rep}_{\mathbbm{k}}(Q)$ is tensor reduced.  
\end{example}

\begin{example}
Let $Q$ be the quiver of type $A_{2}$ that is $Q:=1\longrightarrow 2$, and $\mathcal{A}:=\text{rep}_{\mathbbm{k}}(Q)$. Now consider the functor category $\text{Fun}[\mathcal{A},\mathcal{A}]$ whose monoidal structure is given by the composition of functors (\cite[Example 2.3.12]{etingof2016tensor}). For any object $M\in \mathcal{A}$, since $\mathcal{A}$ is Krull-Schmidt, the decomposition $M\cong S_{1}^{m_{1}}\bigoplus S_{2}^{m_{2}} \bigoplus P_{2}^{m_{3}}$ holds for $m_1, m_2, m_3\in \mathbbm{Z}_{\geqslant 0}$ where
$$S_{1}=(\mathbbm{k}, 0, 0),\;\;S_
{2}=(0, \mathbbm{k}, 0),\;\;P_{2}=(\mathbbm{k}, \mathbbm{k}, \text{id})$$
Now we define an additive functor $F$ as follows:
$$F(S_{1})=F(S_{2})=0,\;F(P_{2})=S_{2},$$
and for any $M,N\in \text{rep}_{k}(Q)$ with $M\cong S_{1}^{m_{1}}\bigoplus S_{2}^{m_{2}} \bigoplus P_{2}^{m_{3}}$, $N\cong S_{1}^{n_{1}}\bigoplus S_{2}^{n_{2}} \bigoplus P_{2}^{n_{3}}$, the $\mathbbm{k}$-linear map
$$\Hom_{\mathcal{A}}(M,N)\longrightarrow \Hom_{\mathcal{A}}(F(M),F(N))$$
is given by
$$
\begin{pmatrix}
m_{1}n_{1}f_{11} & 0 & 0\\ 
0 & m_{2}n_{2}f_{22} & m_{2}n_{3}f_{23}\\
m_{3}n_{1}f_{31} & 0 & m_{3}n_{3}f_{33}
\end{pmatrix}
\mapsto
\begin{pmatrix}
0 & 0 & 0\\ 
0 & m_{2}n_{2}f_{22} & 0\\
0 & 0 & 0
\end{pmatrix},
$$
where $f_{23}\in \Hom_{\mathcal{A}}(S_{2}, P_{2})$, $f_{31}\in \Hom_{\mathcal{A}}(P_{2}, S_{1})$, $f_{33}\in \Hom_{\mathcal{A}}(P_{2}, P_{2})$ and $f_{ii}\in \Hom_{\mathcal{A}}(S_{i}, S_{i})$ when $i\in \{1, 2\}$. Hence $F$ is indeed an additive functor. By definition $F\circ F=0$, which implies that $\text{Fun}[\mathcal{A},\mathcal{A}]$ is not tensor reduced.
\end{example}

\section{Monoidal derived equivalences}\label{Morita/Rickard eq}
In this section we will combine contents of Section \ref{t-stru} and \ref{L.N.T.R.G} to prove the main results. Furthermore, a result on monoidal stable categories is obtained that is
analogous to a theorem of Rickard (see \cite{rickard1989derived}).

\subsection{Monoidal abelian equivalences and monoidal derived equivalences}

Let $\mathcal{A}$ be a locally finite tensor Grothendieck category. According to Example \ref{key example}, both $\textnormal{\Der}^{b}(\textnormal{\fp}(\mathcal{A}))$ and $\textnormal{\Der}^{b}(\mathcal{A})$ are monoidal triangulated categories. In the following, we use $\Sigma$ to denote the shift functors in both $\textnormal{\Der}^{b}(\textnormal{\fp}(\mathcal{A}))$ and $\textnormal{\Der}^{b}(\mathcal{A})$. The aim of this section is to prove the following statement.

\begin{theorem}\emph{(Theorem \ref{Theorem B})}\label{4 eq}
Let $\mathcal{A}$ and $\mathcal{B}$ be locally finite tensor Grothendieck categories. Then the following are equivalent:
\begin{enumerate}
\item $\textnormal{\fp}(\mathcal{A})$ and $\textnormal{\fp}(\mathcal{B})$ are monoidal abelian equivalent;
\item $\mathcal{A}$ and $\mathcal{B}$ are monoidal abelian equivalent;
\item $\textnormal{\Der}^{b}(\textnormal{\fp}(\mathcal{A}))$ and $\textnormal{\Der}^{b}(\textnormal{\fp}(\mathcal{B}))$ are monoidal triangulated equivalent;
\item $\textnormal{\Der}^{b}(\mathcal{A})$ and $\textnormal{\Der}^{b}(\mathcal{B})$ are monoidal triangulated equivalent.

\end{enumerate}
\end{theorem}

We will divide the proof of this statement into three parts. The first step is to show the implications $(1) \bm{\Rightarrow} (3)$ and $(2) \bm{\Rightarrow} (4)$. 

\begin{proposition}\label{1to3 and 2to4}
Let $F:\mathcal{A} \rightarrow \mathcal{B}$ be an exact monoidal functor between monoidal abelian categories with biexact tensor product. Then the derived functor $\textnormal{\Der}^{b}(F)$ of $F$ is a monoidal triangulated functor. In particular, if $F$ is an equivalence, then $\textnormal{\Der}^{b}(F)$ is also an equivalence.

\end{proposition}
\begin{proof}
Since $F$ is exact, $F$ preserves quasi-isomorphisms. Then $F$ extends trivially to $\Der^{b}(F)$ which is a triangulated functor (see \cite[Chapter 10]{weibel1994homological}), namely, for any object $X^{\bullet}=(X^{n},d_{X}^{n})$ in $\Der^{b}(\mathcal{A})$,

$$ \Der^{b}(F)(X^{\bullet})=(F(X^{n}), F(d_{X}^{n})).$$ 

Let $X^{\bullet}$ and $Y^{\bullet}$ be objects in ${\Der}^{b}(\mathcal{A})$. For any $p, q \in \mathbbm{Z}$, as $F$ is a monoidal functor, therefore $F(X^{p}) \otimes F(Y^{q}) \cong F(X^{p} \otimes Y^{q})$.  Then for any $n\in \mathbbm{Z}$,  
\begin{equation*}
\begin{split}
{(\Der^{b}(F)(X^{\bullet})\, \widetilde{\otimes}\, \Der^{b}(F)(Y^{\bullet}))}^{n} &=\coprod_{p+q=n} {F(X^{p})}\otimes {F(Y^{q})} \\
& \cong \coprod_{p+q=n} F(X^{p}\otimes Y^{q}) \\
& \cong F(\coprod_{p+q=n} X^{p}\otimes Y^{q})\\
& = (\Der^{b}(F)(X^{\bullet}\,\widetilde{\otimes}\, Y^{\bullet}))^{n}.
\end{split}
\end{equation*}
Hence $(\Der^{b}(F)(X^{\bullet})\,\widetilde{\otimes}\,\Der^{b}(F)(Y^{\bullet})) \cong \Der^{b}(F)(X^{\bullet}\,\widetilde{\otimes}\, Y^{\bullet})$ which is natural in both $X^{\bullet}$ and $Y^{\bullet}$. In addition, $\textnormal{\Der}^{b}(F)$ preserves the unit. Thus $\textnormal{\Der}^{b}(F)$ is a monoidal triangulated functor. Moreover, if $F$ is an equivalence, then let $G$ denote a quasi-inverse of $F$ and define $\Der^{b}(G)$ in the same way. Then it is clear that $\Der^{b}(G)$ is a quasi-inverse of $\Der^{b}(F)$.
\end{proof}

Next, we give the proof of $(3)\bm{\Rightarrow} (1)$ and $(4) \bm{\Rightarrow} (2)$, starting with two lemmas.

\begin{lemma}\textnormal{(\cite[Example 4.5]{chen2022extensions})}\label{b.on.f.d}
Let $\mathcal{T}$ be a triangulated category and $\mathbbm{t}$ be a bounded t-structure on $\mathcal{T}$. Suppose the heart of $\mathbbm{t}$ is a finite abelian category. Then all bounded t-structures on $\mathcal{T}$ are equivalent to $\mathbbm{t}$. 
  
\end{lemma}

\begin{lemma}\label{bounded in big}

Let $\mathcal{A}$ be a Grothendieck category and $\mathbbm{t}$ be a t-structure on $\textnormal{\Der}^{b}(\mathcal{A})$. Then the following statements are equivalent:
\begin{enumerate}
    \item $\mathbbm{t}$ is bounded; 
    \item $\mathbbm{t}$ is equivalent to the standard t-structure $\mathbbm{t}_{\mathcal{A}}$.
\end{enumerate}
\end{lemma}
\begin{proof}
Let $\Tilde{\mathbbm{t}}_{\mathcal{A}}:=(\Tilde{\mathcal{D}}_{\mathcal{A}}^{\leqslant 0},\Tilde{\mathcal{D}}_{\mathcal{A}}^{\geqslant 1})$ and $\mathbbm{t}_{\mathcal{A}}:=(\mathcal{D}_{\mathcal{A}}^{\leqslant 0},\mathcal{D}_{\mathcal{A}}^{\geqslant 1})$ be the standard t-structures on $\Der(\mathcal{A})$ and $\Der^{b}(\mathcal{A})$ respectively. It is clear that $(2)\Rightarrow (1)$, since the standard t-structure $\mathbbm{t}_{\mathcal{A}}$ is bounded. Next, we prove the other implication. 

\par Let $A$ be a generator of $\mathcal{A}$ and $E$ be the injective envelope of $A$. Then $E$ is an injective cogenerator of $\mathcal{A}$. The assumption that $\mathbbm{t}$ is bounded implies that the injective cogenerator $E$ lies in $\mathcal{D}^{\geqslant m+1}$ for some $m\in \mathbbm{Z}$. Due to \cite[Remark 4.11]{psaroudakis2018realisation}, for any object $X\in \Der(\mathcal{A})$, $X\in \Tilde{\mathcal{D}}_{\mathcal{A}}^{\leqslant 0}$ if and only if $\Hom_{\Der(\mathcal{A})}(X,E[i])=0$ for all $i\leqslant -1$. Let $X\in \mathcal{D}^{\leqslant m}$. Then $\Hom_{\Der(\mathcal{A})}(X,\mathcal{D}^{\geqslant m+1})=0$. Since $\mathcal{D}^{\geqslant m+1}$ is closed under negative shifts, $\{E[i]\}_{i\leqslant -1}$ are objects in $\mathcal{D}^{\geqslant m+1}$. Hence $X\in \Tilde{\mathcal{D}}_{\mathcal{A}}^{\leqslant 0}\cap \Der^{b}(\mathcal{A})=\mathcal{D}_{\mathcal{A}}^{\leqslant 0}$, which yields that $\mathcal{D}^{\leqslant m} \subseteq \mathcal{D}_{\mathcal{A}}^{\leqslant 0}$.

\par The boundedness of $\mathbbm{t}$ tells us that $A \in \mathcal{D}^{\leqslant n} $ for some $n \in \mathbbm{Z}$. We fix the following notations in $\Der(\mathcal{A})$:
$$\mathcal{U}:={({\mathcal{D}}^{\leqslant n})}^{\perp}:=\{U\in \Der(\mathcal{A})\;|\; \Hom_{\Der(\mathcal{A})}(Y,U)=0,\;\forall Y \in \mathcal{D}^{\leqslant n}\},$$
$${^{\perp}(\mathcal{U})}:=\{V\in \Der(\mathcal{A})\;|\; \Hom_{\Der(\mathcal{A})}(V,U)=0,\;\forall U \in \mathcal{U}\}.$$
In fact, ${^{\perp}(\mathcal{U})}$ is a suspended subcategory (i.e. it is closed under extensions and positive iterations of the shift functor) and closed under coproducts in $\Der(\mathcal{A})$. The definition of ${^{\perp}(\mathcal{U})}$ implies that $A\in \mathcal{D}^{\leqslant n}\subseteq {^{\perp}(\mathcal{U})}$. \cite[Lemma 4.10]{psaroudakis2018realisation} shows that $\Tilde{\mathcal{D}}_{\mathcal{A}}^{\leqslant 0}$ is the smallest suspended subcategory closed under coproducts in $\Der(\mathcal{A})$, which is containing $A$. Thus $\Tilde{\mathcal{D}}_{\mathcal{A}}^{\leqslant 0}\subseteq {^{\perp}(\mathcal{U})}$. By taking the restriction,
$$\mathcal{D}_{\mathcal{A}}^{\leqslant 0}=\Tilde{\mathcal{D}}_{\mathcal{A}}^{\leqslant 0}\cap \Der^{b}(\mathcal{A})\subseteq {^{\perp}(\mathcal{U})}\cap \Der^{b}(\mathcal{A}).$$
We claim that ${^{\perp}(\mathcal{U})}\cap \Der^{b}(\mathcal{A})=\mathcal{D}^{\leqslant n}$, then we finish the proof. Indeed, ${^{\perp}(\mathcal{U})}\cap \Der^{b}(\mathcal{A}) \supseteq \mathcal{D}^{\leqslant n}$ is clear, we only show another implication. Let $X \in {^{\perp}(\mathcal{U})}\cap \Der^{b}(\mathcal{A})$. Note that $\mathcal{D}^{\geqslant n+1}\subseteq \mathcal{U}$, hence $\Hom_{\Der(\mathcal{A})}(X,\mathcal{D}^{\geqslant n+1})=0$ and then $X\in \mathcal{D}^{\leqslant n}$.
\end{proof}

\begin{proposition}\label{3to1 and 4to2}
Let $\mathcal{A}$ and $\mathcal{B}$ be locally finite tensor Grothendieck categories. Then the following statements hold.
\begin{enumerate}
\item If $\textnormal{\Der}^{b}(\textnormal{\fp}(\mathcal{A}))$ and $\textnormal{\Der}^{b}(\textnormal{\fp}(\mathcal{B}))$ are monoidal triangulated equivalent, then $\textnormal{\fp}(\mathcal{A})$ and $\textnormal{\fp}(\mathcal{B})$ are monoidal abelian equivalent.
\item If $\textnormal{\Der}^{b}(\mathcal{A})$ and $\textnormal{\Der}^{b}(\mathcal{B})$ are monoidal triangulated equivalent, then $\mathcal{A}$ and $\mathcal{B}$ are monoidal abelian equivalent.
\end{enumerate}
\end{proposition}
\begin{proof}
Let $\mathbbm{t}_{\fp(\mathcal{A})}$ and $\mathbbm{t}_{\mathcal{A}}$ (resp. $\mathbbm{t}_{\fp(\mathcal{B})}$ and $\mathbbm{t}_{\mathcal{B}}$) be the standard t-structures on $\textnormal{\Der}^{b}(\textnormal{\fp}(\mathcal{A}))$ and $\textnormal{\Der}^{b}(\mathcal{A})$ (resp. $\textnormal{\Der}^{b}(\textnormal{\fp}(\mathcal{B}))$ and $\textnormal{\Der}^{b}(\mathcal{B})$). By Proposition \ref{trivial ideal} and Theorem \ref{0tensor lemma}, $\mathbbm{t}_{\fp(\mathcal{A})}$ and $\mathbbm{t}_{\mathcal{A}}$ (resp. $\mathbbm{t}_{\fp(\mathcal{B})}$ and $\mathbbm{t}_{\mathcal{B}}$) are tensor reduced, which yields that $\textnormal{\Der}^{b}(\textnormal{\fp}(\mathcal{A}))$ and $\textnormal{\Der}^{b}(\mathcal{A})$ (resp. $\textnormal{\Der}^{b}(\textnormal{\fp}(\mathcal{B}))$ and $\textnormal{\Der}^{b}(\mathcal{B})$) are tensor reduced. Suppose that the functors
$$F: \textnormal{\Der}^{b}(\textnormal{\fp}(\mathcal{A}))\rightarrow \textnormal{\Der}^{b}(\textnormal{\fp}(\mathcal{B})),\;\Tilde{F}:\textnormal{\Der}^{b}(\mathcal{A}) \rightarrow \textnormal{\Der}^{b}(\mathcal{B})$$ 
are monoidal triangulated equivalent. By Corollary \ref{unique u.t.e cor} (1), there are monoidal t-structures $\mathbbm{t}:=(\mathcal{U},\mathcal{V})$ and $\Tilde{\mathbbm{t}}:=(\Tilde{\mathcal{U}},\Tilde{\mathcal{V}})$ on ${\Der}^{b}(\fp(\mathcal{B}))$ and $\textnormal{\Der}^{b}(\mathcal{B})$ respectively, where
\begin{equation*}
\begin{split}
&\mathcal{U}=\{X^{\bullet}\in {\Der}^{b}(\fp(\mathcal{B}))\;|\; \exists V^{\bullet} \in  \mathcal{D}_{\fp(\mathcal{A})}^{\leqslant 0} \text{ such that } F(V^{\bullet}) \cong X^{\bullet}\},\\
&\mathcal{V}=\{Y^{\bullet}\in {\Der}^{b}(\fp(\mathcal{B}))\;|\; \exists W^{\bullet} \in  \mathcal{D}_{\fp(\mathcal{A})}^{\geqslant 1} \text{ such that } F(W^{\bullet}) \cong Y^{\bullet}\},\\ 
&\Tilde{\mathcal{U}}=\{X^{\bullet}\in \textnormal{\Der}^{b}(\mathcal{B})\;|\; \exists V^{\bullet} \in  \mathcal{D}_{\mathcal{A}}^{\leqslant 0} \text{ such that } \Tilde{F}(V^{\bullet}) \cong X^{\bullet}\}, \\
&\Tilde{\mathcal{V}}=\{Y^{\bullet}\in \textnormal{\Der}^{b}(\mathcal{B})\;|\; \exists W^{\bullet} \in  \mathcal{D}_{\mathcal{A}}^{\geqslant 1} \text{ such that } \Tilde{F}(W^{\bullet}) \cong Y^{\bullet}\}.
\end{split}
\end{equation*}
According to Lemma \ref{b.on.f.d} and Lemma \ref{bounded in big}, it follows that $\mathbbm{t}$ and $\Tilde{\mathbbm{t}}$ are equivalent to $\mathbbm{t}_{\fp(\mathcal{B})}$ and $\mathbbm{t}_{\mathcal{B}}$ respectively. The proposition follows then from Corollary \ref{unique u.t.e cor} (2).
\end{proof}

\begin{proof}[Proof of Theorem \ref{4 eq}]
Following Proposition \ref{Morita tensor eq}, we obtain that $(1)\Leftrightarrow (2)$. Proposition \ref{1to3 and 2to4} tells us that $(1)\Rightarrow (3)$ and $(2) \Rightarrow (4)$. Proposition \ref{3to1 and 4to2} shows that $(3)\Rightarrow (1)$ and $(4) \Rightarrow (2)$.
\end{proof}
\begin{remark} (Corollary \ref{Corollary A})\label{corollary A}
Owing to Example \ref{hopf morita} and Remark \ref{gauge}, a special case of Theorem \ref{4 eq} is the following version for Hopf algebras.
Let $H$ and $H'$ be finite-dimensional Hopf algebras. Then the following are equivalent:
\begin{enumerate}
\item $H$ and $H'$ are gauge equivalent;
\item $\mods\-H$ and $\mods\-H'$ are monoidal abelian equivalent;
\item $\Mod\-H$ and $\Mod\-H'$ are monoidal abelian equivalent;
\item $\textnormal{\Der}^{b}(\mods\-H)$ and $\textnormal{\Der}^{b}(\mods\-H')$ are monoidal triangulated equivalent;
\item $\textnormal{\Der}^{b}(\Mod\-H)$ and $\textnormal{\Der}^{b}(\Mod\-H')$ are monoidal triangulated equivalent.

\end{enumerate}

More generally, for a finite-dimensional weak quasi-Hopf algebra $H$ (see \cite[Section 2.5]{haring1997reconstruction}), $\Mod\-H$ is a locally finite tensor Grothendieck category as $\textnormal{\fp}(\Mod\-H)=\mods\-H$ is a finite tensor category \cite{haring1997reconstruction}. Thus Theorem \ref{4 eq} also holds for finite-dimensional weak quasi-Hopf algebras.
\end{remark}

\begin{remark}\label{locally finite ring Grothendieck categories}

Let $Q$ be a finite acyclic quiver. According to Example \ref{rep exacmple}, $\text{rep}_{\mathbbm{k}}Q$ is a tensor reduced monoidal abelian category with biexact tensor product. And Example \ref{key example} tells us that the standard t-structure on $\Der^{b}(\text{rep}_{\mathbbm{k}}Q)$ is a monoidal t-structure. Then Corollary \ref{unique u.t.e cor} implies the following statement.   
\end{remark}

\begin{corollary}\textnormal{(\cite[Corollary 0.6]{zhang2022frobenius})}\label{zhang}
Let $Q$ and $Q'$ be two finite acyclic quivers such that $\emph{\Der}^{b}(\emph{rep}_{\mathbbm{k}}Q)$ and $\emph{\Der}^{b}(\emph{rep}_{\mathbbm{k}}Q')$ are monoidal triangulated equivalent, then $Q$ and $Q'$ are isomorphic as quivers. 
\end{corollary}

We emphasize that in the proof of Proposition \ref{3to1 and 4to2} the condition of monoidal equivalences is essential for this process, as shown in the following example.

\begin{example}
Consider the following quivers of type $A_5$: 
$$Q: 1\longrightarrow 2 \longrightarrow  3\longrightarrow  4\longrightarrow 5,$$
$$Q':1 \longleftarrow 2 \longleftarrow 3 \longrightarrow 4 \longrightarrow 5.$$
$\Der^{b}(\text{rep}_{\mathbbm{k}}Q)$ and $\Der^{b}(\text{rep}_{\mathbbm{k}}Q')$ are derived equivalent according to \cite[Section 5.2.8]{angeleri2007handbook}, whereas $\text{rep}_{\mathbbm{k}}Q$ and $\text{rep}_{\mathbbm{k}}Q'$ are not equivalent as abelian categories. Moreover, Corollary \ref{zhang} tells us that $\Der^{b}(\text{rep}_{\mathbbm{k}}Q)$ and $\Der^{b}(\text{rep}_{\mathbbm{k}}Q')$ are not monoidal derived equivalent. 
\end{example}

\subsection{Monoidal stable equivalences induced by monoidal derived equivalences}
In retrospect, all finite tensor categories with enough projectives are Frobenius categories (\cite[Proposition 2.3]{etingof2004finite}). In addition, the tensor product of any object with a projective object in a tensor category is still projective (\cite[Corollary 2, p.441]{kazhdan1994tensor}). Therefore, the stable categories of finite tensor categories with enough projectives are monoidal triangulated categories (\cite[Section 1.2]{nakano2024spectrum}), which are called \textbf{monoidal stable categories} in this paper.

The purpose of this subsection is to prove an analogue for a monoidal triangulated category of Rickard's result describing a stable category as a Verdier quotient of a derived category. We use $\mathcal{P}$ to denote the full subcategory of an abelian category, which consists of projective objects.

\begin{theorem}\emph{(\cite[Theorem 2.1]{rickard1989derived})}\label{Ric89_1}
Let $\mathcal{A}$ be a finite Frobenius abelian category. The essential image of the natural embedding $$\textnormal{\textbf{K}}^b(\mathcal{P})\rightarrow \textnormal{\Der}^b(\mathcal{A})$$ is a thick subcategory. The Verdier quotient category $\textnormal{\Der}^b(\mathcal{A})/\textnormal{\textbf{K}}^b(\mathcal{P})$ is equivalent as a triangulated category to the stable category $\underline{\mathcal{A}}$.
    
\end{theorem}

In view of \cite[Corollary 8.3]{rickard1989morita} and Theorem \ref{Ric89_1}, if two finite Frobenius abelian categories $\mathcal{A}$ and $\mathcal{A}'$ are derived equivalent, then they are stably equivalent (\cite[Corollary 2.2]{rickard1989derived}). 
According to \cite[Definition 1.2]{balmer2005spectrum}, a \textbf{thick tensor ideal} $\mathcal{I}$ of a monoidal triangulated category $\mathcal{T}$ is a thick subcategory such that 
$\mathcal{I}$ is a tensor ideal. 

\begin{lemma}\emph{(\cite[Remark 4.0.6]{pauwels2015quasi})}\label{quotient}
    Let $\mathcal{I}$ be a thick tensor ideal of a monoidal triangulated category $\mathcal{T}$. Then the Verdier quotient category $\mathcal{T}/\mathcal{I}$ is still a monoidal triangulated category.
\end{lemma}
Firstly, Theorem \ref{Ric89_1} can be realized in the case of monoidal derived categories.
\begin{lemma}\label{tensor ideal}
    Let $\mathcal{A}$ be a finite tensor category with enough projective objects. The essential image of the natural embedding $$\textnormal{\textbf{K}}^b(\mathcal{P})\rightarrow \textnormal{\Der}^b(\mathcal{A})$$ is a thick tensor ideal. Moreover, $\textnormal{\Der}^b(\mathcal{A})/\textnormal{\textbf{K}}^b(\mathcal{P})$ is a monoidal triangulated category. 
\end{lemma}
\begin{proof}
    Let $\mathcal{I}$ be the essential image of the natural embedding. We already know that $\mathcal{I}$ is a thick subcategory of $\Der^{b}(\mathcal{A})$ by Theorem \ref{Ric89_1}. The only thing we need to verify is that $\mathcal{I}$ is a tensor ideal. Let $P^{\bullet}\in \mathcal{I}$ and $Q^{\bullet}\in \Der^{b}(\mathcal{A})$. According to Theorem \ref{K.lemma}, $Q^{\bullet}\widetilde{\otimes} P^{\bullet}$ is quasi-isomorphic to a complex $G^{\bullet}\in \textbf{K}^b(\mathcal{P})$, hence $P^{\bullet}\widetilde{\otimes} Q^{\bullet} \in \mathcal{I}$. Similarly, we have $Q^{\bullet}\widetilde{\otimes} P^{\bullet} \in \mathcal{I}$. Applying Lemma \ref{quotient} completes the proof.  
\end{proof}

The definition of a monoidal functor implies the following result.
\begin{lemma}\label{composite monoidal}
Let $\mathcal{M}$, $\mathcal{N}$ and $\mathcal{C}$ be monoidal categories. Consider the following commutative diagram
$$\begin{tikzcd}
{\mathcal{M}} && {\mathcal{C}} \\
& {\mathcal{N}}
\arrow["F", from=1-1, to=1-3]
\arrow["G"', from=1-1, to=2-2]
\arrow["H"', from=2-2, to=1-3]
\end{tikzcd}$$
where $F$, $G$ are monoidal functors and $H$ is any functor. If $G$ is dense, then $H$ is also a monoidal functor.
\end{lemma}

\begin{proposition}\label{Verdier tensor equi to stable}
    Let $\mathcal{A}$ be a finite tensor category with enough projective objects. The category $\textnormal{\Der}^b(\mathcal{A})/\textnormal{\textbf{K}}^b(\mathcal{P})$ is equivalent as a monoidal triangulated category to the monoidal stable category $\underline{\mathcal{A}}$.
\end{proposition}
\begin{proof}
    
By Theorem \ref{Ric89_1} and Lemma \ref{tensor ideal}, it remains to prove that the equivalence  $$F:\underline{\mathcal{A}}\rightarrow \Der^b(\mathcal{A})/\textbf{K}^b(\mathcal{P})$$ in the proof of Theorem \ref{Ric89_1} is also a monoidal equivalence. Recall that $F$ is given by the following diagram:
$$
\begin{tikzcd}
F':\;{\mathcal{A}} \arrow[r, tail] \arrow[rd] & {\Der^b(\mathcal{A})} \arrow[r, two heads] & {\Der^b(\mathcal{A})/\textbf{K}^b(\mathcal{P})} \\
                              & {\underline{\mathcal{A}}} \arrow[ru,"F"']           &   
\end{tikzcd}
$$
where $F'$ is obtained by composing the natural embedding of $\mathcal{A}$ into $\Der^b(\mathcal{A})$ with the Verdier functor. Since the embedding functor and the Verdier functor are monoidal functors, $F'$ is a monoidal functor. By Lemma \ref{composite monoidal}, $F$ is also a monoidal functor and we finish the proof.
\end{proof}

\begin{remark}\label{fact2}
Let $F: \mathcal{T}\rightarrow\mathcal{T}'$ be a monoidal triangulated functor between two monoidal triangulated categories $\mathcal{T}$ and $\mathcal{T}'$. Suppose that $\mathcal{I}$, $\mathcal{I}'$ are thick tensor ideals of $\mathcal{T}$, $\mathcal{T}'$ respectively such that $F(\mathcal{I})\subset \mathcal{I}'$. Hence there exists a triangulated functor such that the following diagram commutes. 

\[\begin{tikzcd}[ampersand replacement=\&]
	\mathcal{T} \&\& {\mathcal{T}'} \\
	{\mathcal{T}/\mathcal{I}} \&\& {\mathcal{T}'/\mathcal{I}'}
	\arrow["F", from=1-1, to=1-3]
	\arrow[from=1-1, to=2-1]
	\arrow[from=1-3, to=2-3]
	\arrow["{\underline{F}}"', from=2-1, to=2-3]
\end{tikzcd}\]
By Lemma \ref{composite monoidal}, $\underline{F}$ is also a monoidal functor. If, moreover, $F$ is an equivalence and $F(\mathcal{I})\simeq \mathcal{I}'$, then $\underline{F}$ is also an equivalence.
\end{remark}

Using Proposition \ref{Verdier tensor equi to stable} and Remark \ref{fact2} or directly as a corollary of Theorem \ref{4 eq}, we obtain the monoidal triangulated categories' version of \cite[Corollary 2.2]{rickard1989derived}.
\begin{corollary}\label{result4}  
Let $\mathcal{A}$ and $\mathcal{A}'$ be finite tensor categories with enough projective objects. If $\textnormal{\Der}^b(\mathcal{A})$ and $ \textnormal{\Der}^b(\mathcal{A}')$ are equivalent as monoidal triangulated categories, then $$\underline{\mathcal{A}}\simeq\underline{\mathcal{A}'}$$ as monoidal triangulated categories.
\end{corollary}

\begin{remark}
We consider a special case of Corollary \ref{result4}. Let $H$ and $H'$ be finite-dimensional Hopf algebras. If $\textnormal{\Der}^b(\mods\-H)$ and $\textnormal{\Der}^b(\mods\-H')$ are equivalent as monoidal triangulated categories, then $$\underline{\mods}\-H \simeq \underline{\mods}\-H'$$ as monoidal triangulated categories.
\end{remark}

\begin{example}
We end this section by pointing out that there are Hopf algebras being only stable equivalent but not monoidal stable equivalent. Consider the $n^2$-dimensional Taft algebras $H_{n}(q_1)$ and $H_{n}(q_2)$, where $q_1$ and $q_2$ are primitive $n$-th roots of unity (see definition in \cite[Chapter 7.3]{radford2012hopf}). $H_{n}(q_1)$ and $H_{n}(q_2)$ are gauge equivalent if and only if $q_1=q_2$ \cite[Corollary 3.3]{kashina2012trace}. As $H_{n}(q_1)$ and $H_{n}(q_2)$ are isomorphic as algebras, they are Morita equivalent inducing a stable equivalence from $\underline{\mods}\-H_{n}(q_1)$ to $\underline{\mods}\-H_{n}(q_2)$. According to \cite[Corollary 3.2.13]{xu2024morita}, this stable equivalence can not be a monoidal stable equivalence if $q_1\neq q_2$.
\end{example}

\section*{Acknowledgements} 
The authors are deeply grateful to Steffen K\"onig and Frederik Marks for inspiring discussions and comments. They also would like to thank Gongxiang Liu for providing research topics and useful remarks, thank Simone Virili for helpful conversations on locally coherent Grothendieck categories, and thank Jorge Vit\'oria for valuable discussions on  Lemma \ref{bounded in big}.

\phantomsection
\bibliographystyle{plain}
\bibliography{refs.bib}
\end{document}